\documentclass{article}
\usepackage[utf8]{inputenc}
\usepackage{amsthm}
\usepackage{amsmath}
\usepackage{amsfonts}
\usepackage{latexsym}
\usepackage{geometry}
\usepackage{enumerate}
\usepackage{csquotes}
\usepackage{graphicx}
\usepackage{comment}
\usepackage{appendix}
\usepackage{hyperref}
\usepackage{bm}

\emergencystretch=\maxdimen
\def \W {\mathcal{W}}
\def \N {\mathcal{N}}
\def \PP {\mathcal{P}}
\def \XX {\mathcal{X}}

\def \dist {{\rm dist}}
\DeclareMathOperator{\Rm}{Rm}
\DeclareMathOperator{\Ric}{Ric}
\DeclareMathOperator{\Vol}{Vol}

\DeclareMathOperator{\tr}{tr}
\makeatletter
\newcommand*{\rom}[1]{\rm {\expandafter\@slowromancap\romannumeral #1@}}
\makeatother

\def \tg {\tilde{g}}
\def \bg {\overline{g}}

\def \R {\mathfrak{R}}
\def \V {\mathcal{V}}
\def \div {{\rm div}}

\numberwithin{equation}{section}
\newtheorem{Theorem}{Theorem}[section]
\newtheorem{Proposition}[Theorem]{Proposition}
\newtheorem{Lemma}[Theorem]{Lemma}
\newtheorem{Corollary}[Theorem]{Corollary}
\theoremstyle{definition}

\title{On Ricci flows with closed and smooth tangent flows}
\author{Pak-Yeung Chan, Zilu Ma, and Yongjia Zhang}

\begin{document}

\maketitle

\begin{abstract}
    In this paper, we consider Ricci flows admitting closed and smooth tangent flows in the sense of Bamler \cite{Bam20c}. The tangent flow in question can be either a tangent flow at infinity for an ancient Ricci flow, or a tangent flow at a singular point for a Ricci flow developing a finite-time singularity. Among other things, we prove: (1) that in these cases the tangent flow must be unique, 
    (2) that if a Ricci flow with finite-time singularity has a closed singularity model, then the singularity is of Type I and the singularity model is the tangent flow at the singular point; this answers a question proposed in \cite{RFV3}, (3) a dichotomy theorem that characterizes ancient Ricci flows admitting a closed and smooth backward sequential limit.
\end{abstract}
\setcounter{tocdepth}{1}
\tableofcontents

\section{Introduction}

In \cite{Bam20a}---\cite{Bam20c}, Bamler has established a theory for the weak limits ($\mathbb{F}$-limits) of Ricci flows on closed manifolds. These limit spaces are metric flows called \emph{singular spaces}. The theory of singular spaces bears many similarities with several other geometric analysis theories such as the minimal surface, the Brakke flow, and the Cheeger-Colding theory. Among many other things, one important result proved in \cite{Bam20c} is the codimension-$4$ theorem for the singular space. Similar to the theorem of Cheeger-Naber \cite{CN15}, the singularity of a singular space also has parabolic (space-time) Minkowski codimension no smaller than $4$ \cite[Theorem 2.3]{Bam20c}.

Another important theme in the study of Bamler's singular space is the uniqueness of \emph{tangent flow}, as in the cases of the mean curvature flow, the minimal surface, and the Cheeger-Colding theory (in the latter two cases, the counterpart of tangent flow is called \emph{tangent cone}). In the field of minimal surface, White \cite{W83} proved that a $2$-dimensional minimizing current has unique tangent cone everywhere, Allard-Almgren \cite{AA81} proved the uniqueness of tangent cone under some integrability condition for the cross section, Simon \cite{S83} proved that if a tangent cone at a point on a minimal surface is smooth, then it is the unique tangent cone at that point. In Cheeger-Colding theory, Colding-Minicozzi \cite{CM14} proved that a tangent cone at a point on a 
noncollapsed
Ricci limit space is the unique tangent cone at that point if the cone has 
a smooth
cross section. In the mean curvature flow, Colding-Minicozzi \cite{CM15} proved that if a tangent flow at a point is a cylinder, then it is the unique tangent flow at that point. We would like to draw the reader's attention to a recent post of Colding-Minicozzi \cite{CM21}. Their method, as is claimed in that paper, can be applied to show that cylindrical singularities in the Ricci flow are isolated singularities, and they will prove that this type of blow-up limits must be unique in a forthcoming paper. 

In all these fields, it is also interesting to consider the tangent cone/flow at infinity of a noncompact minimal surface/Ricci limit space or an ancient mean curvature flow/singular space. For instance, Cheeger-Tian \cite{CT94} proved the uniqueness of tangent cone at infinity for Ricci flat manifolds with Euclidean volume growth under certain curvature
and integrability assumptions. This result was later strengthened by
Colding-Minicozzi \cite{CM14}. For an ancient singular space, Bamler \cite{Bam20c} has considered the tangent flow at infinity and has shown that it is, as in the case of the tangent flow at one point, a shrinking metric soliton. The authors \cite{CMZ21} proved that if Bamler's tangent flow at infinity is smooth, then it is not different from Perelman's asymptotic shrinking gradient Ricci soliton.

In this article, we shall study both the tangent flow at infinity for an ancient Ricci flow and the tangent flow at a singular point for a Ricci flow developing a finite-time singularity. One of our main results is that, in either of the two cases above, if the tangent flow is smooth and closed, then this tangent flow is the unique one, and hence independent of the choice of the scaling factors or the base point (in the case of a Ricci flow forming a finite-time singularity, the base point should really be understood as the choice of the \emph{singular conjugate heat kernel}; see Section 2 for the definition).

Let us consider a metric flow $\displaystyle\left(\XX,\mathfrak{t},(\dist_t)_{t\in I},(\nu_{x\,|\, s})_{x\in \XX,s\in I,s\leq\mathfrak{t}(x)}\right)$  over $I$ induced by a smooth Ricci flow $(M^n,g(t))_{t\in I}$ (see \cite[Definition 3.2]{Bam20b}), where $M$ is a closed manifold and $I\subset \mathbb{R}$ is an interval. For any $\lambda>0$, we shall use the notation $\XX^{0,\lambda}$ to represent the scaled metric flow:
$$\XX^{0,\lambda}=\left(\XX,\lambda^{2}\mathfrak{t},(\lambda\dist_{\lambda^{-2}t})_{t\in \lambda^2 I},(\nu_{x\,|\, \lambda^{-2}s})_{x\in \XX,s\in \lambda^2 I,s\leq\mathfrak{t}(x)}\right),\quad t\in \lambda^2 I.$$
\begin{enumerate}[(1)]
    \item If $(M,g(t))$ is ancient, then we assume that $I=(-\infty,0]$ and hence $\XX=M^n\times(-\infty,0]$. Let us fix a point $(p_0,0)\in M\times\{0\}=\XX_0$ and let $\mu_t:=\nu_{p_0,0\,|\,t}$ be the conjugate heat kernel based at $(p_0,0)$. By Bamler's compactness theorems \cite{Bam20b}, for any sequence $\{\tau_i\}_{i=1}^\infty$ with $\tau_i\nearrow\infty$, we have
    \begin{eqnarray}\label{the_F_convergence}
\left(\XX^{0,(\tau_i)^{-\frac{1}{2}}},\left(\mu_{\tau_it}\right)_{t\in(-\infty,0]}\right)
\xrightarrow{\makebox[1cm]{$\mathbb{F}$}}
\left(\XX^\infty,(\mu^\infty_t)_{t\in -(\infty,0)}\right)
\end{eqnarray}
after passing to a subsequence. Here $\XX^\infty$ is a metric flow over $(-\infty,0]$. If the original ancient solution $(M,g(t))$ has uniformly bounded Nash entropy, that is, 
\begin{eqnarray}\label{Nash_bound}
\text{there exist $(p,t)\in M^n\times(-\infty,0]$ and  $Y<\infty$ such that}\quad \N_{p,t}(\tau)\geq -Y\text{ for all }\tau>0,
\end{eqnarray}
then the singular set of $\XX^\infty$ has parabolic Minkowski codimension no smaller than $4$, and $\left(\XX^\infty,(\mu^\infty_t)_{t\in -(\infty,0)}\right)$ is a metric soliton; see \cite{Bam20c}.
\item If $(M,g(t))$ has finite-time singularity, then we assume that $I=[-T,0)$ and $t=0$ is the singular time. In this case, the metric flow $\XX$ is not defined at $t=0$. We will let $(\mu_t)_{t\in I}$ be a singular conjugate heat kernel based at $t=0$ as defined in \cite{MM15}. By Bamler's compactness theorem \cite{Bam20b} again, for any sequence $\{\tau_i\}_{i=1}^\infty$ with $\tau_i\searrow 0$, we also have that (\ref{the_F_convergence}) holds after passing to a subsequence. In this case, the sequence is automatically noncollapsed, and the limit $\left(\XX^\infty,(\mu^\infty_t)_{t\in -(\infty,0)}\right)$ is a metric soliton whose singular set has parabolic Minkowski codimension no smaller than $4$ \cite{Bam20c}.
\end{enumerate}

Our first main results are:

\begin{Theorem}[Uniqueness of smooth and closed tangent flow at infinity]\label{Thm_main}
Let $(M^n,g_t)_{t\in(-\infty,0]}$ be an ancient solution satisfying (\ref{Nash_bound}), where $M$ is a closed manifold. Assume that there is a sequence of scaling factors $\{\tau_i\}_{i=1}^\infty$ with $\tau_i\nearrow\infty$, such that the limit metric flow $\XX^\infty$ in (\ref{the_F_convergence}) is a smooth Ricci flow on a closed manifold. Then for any sequence $\{\tau_i'\}_{i=1}^\infty$ with $\tau_i'\nearrow\infty$ and for any $(p_1,t_1)\in M\times(-\infty,0]$, we have
\[
\left((M,\tau_i'^{-1}g_{\tau_i' t})_{t\in -(\infty,0]},(\nu_{p_1,t_1|\tau_i' t})_{t\in -(\infty,0]}\right)
\xrightarrow{\makebox[1cm]{$\mathbb{F}$}}
\left(\XX^\infty,(\mu^\infty_t)_{t\in -(\infty,0)}\right).
\]
In other words, the tangent flow at infinity is independent of the sequence of scaling factors and the base point.
\end{Theorem}

\begin{Theorem}[Uniqueness of smooth and closed tangent flow at singular point]\label{Thm_main_1}
Let $(M^n,g_t)_{t\in[-T,0)}$ be a Ricci flow on a closed manifold $M$ which develops a finite-time singularity at $t=0$. Assume that there exist a sequence $\{\tau_i\}_{i=1}^\infty$ with $\tau_i\searrow 0$ such that the limit flow $\XX^\infty$ in (\ref{the_F_convergence}) is a smooth Ricci flow on a closed manifold. Then for any sequence $\{\tau_i'\}_{i=1}^\infty$ with $\tau_i'\searrow 0$ and any singular conjugate heat kernel $(\mu'_t)_{t\in[-T,0)}$, we have
\[
\left((M,\tau_i'^{-1}g_{\tau_i' t})_{t\in [-T\tau_i'^{-1},0)},(\mu'_{\tau_i' t})_{t\in[-T\tau_i'^{-1},0)}\right)
\xrightarrow{\makebox[1cm]{$\mathbb{F}$}}
\left(\XX^\infty,(\mu^\infty_t)_{t\in -(\infty,0)}\right).
\]
\end{Theorem}

In regard of the ancient Ricci flow in Theorem \ref{Thm_main}, one may rightly ask whether the ancient solution in question is the same as the asymptotic soliton itself. In fact, this question is answered by Kr\"oncke \cite{Kr15}: it depends on the stability of the asymptotic soliton. If the closed asymptotic soliton is stable, then the ancient Ricci flow must be the same as this soliton. However, there exists nontrivial (that is, non-self-similar) ancient solution flowing out of an unstable closed soliton. For other results related to Theorem \ref{Thm_main_1}, refer to \cite{Ac12} and \cite{Se06}. 

Our major technique of proving the above two results is a \L ojaciewicz inequality proved by Sun-Wang \cite{SW15}. This \L ojaciewicz inequality shows that, once the Riemannian metric evolving by a \emph{modified} (forward or backward) Ricci flow becomes too close to a shrinking gradient Ricci soliton, then it will stay close to this soliton for all the time; if this soliton is also a sequential limit, then the modified flow converges to the soliton at polynomial rate. The following corollaries follow immediately from this technique.

\begin{Corollary}\label{Coro_main}
Under the assumptions of Theorem \ref{Thm_main}, the limit flow $\left(\XX^\infty,(\mu^\infty_t)_{t\in -(\infty,0)}\right)$ is the canonical form of a shrinking gradient Ricci soliton $(M^n,g_o,f_o)$, where $\mu_t^\infty$ is the conjugate heat flow constructed using $f_o$. Furthermore, the modified backward Ricci flow converges to $g_o$ at polynomial rate. More precisely, for all $t\ll -1$, there is a self-diffmorphism $\phi_t:M^n\to M^n$, evolving smoothly in $t$, such that
$$\left\|\,\tfrac{1}{|t|}\phi_t^*g_t-g_o\right\|_{C^{k,\gamma}_{g_o}}\leq C(\log(-t))^{-\beta},$$ where $k\gg 1$, $\gamma, \beta\in (0,1)$, and $C<\infty$.
\end{Corollary}

\begin{Corollary}\label{Coro_main_1}
Under assumptions of Theorem \ref{Thm_main_1}, the limit flow $\left(\XX^\infty,(\mu^\infty_t)_{t\in -(\infty,0)}\right)$ is the canonical form of a shrinking gradient Ricci soliton $(M^n,g_o,f_o)$, where $\mu^\infty_t$ is the conjugate heat flow constructed using $f_o$. Furthermore, the modified Ricci flow converges to $g_o$ at a polynomial rate. More precisely, for all $t$ close enough to $0$, there is a self-diffmorphism $\phi_t:M^n\to M^n$, evolving smoothly in $t$, such that
$$\left\|\,\tfrac{1}{|t|}\phi_t^*g_t-g_o\right\|_{C^{k,\gamma}_{g_o}}\leq C(-\log(-t))^{-\beta},$$ where $k\gg 1$, $\gamma, \beta\in (0,1)$, and $C<\infty$.
\end{Corollary}

A more general and difficult question is whether one can prove the uniqueness of the tangent flow (either at a point or at infinity) without the assumptions in Theorem \ref{Thm_main} or Theorem \ref{Thm_main_1}. In fact, by definition, it is easy to observe that the \emph{uniqueness of tangent flow at infinity} consists of two statements: (1) that the tangent flow is independent of the base point, (2) that the tangent flow is independent of the sequence of scaling factors. As mentioned earlier, Bamler's tangent flow at infinity resembles the asymptotic cone of some static spaces. It is well-known that the asymptotic cone of a boundedly compact metric space does not depend on the base point; see, e.g., \cite[Proposition 8.2.8]{BBI01}. 
Here we present a parallel result, which also generalizes the authors' previous result \cite[Theorem 8.1]{CMZ21}. Due to the technicality of the $\mathbb{F}$-convergence, the proof of the this theorem is significantly more complicated than the parallel case of asymptotic cone, although they are in the same spirit. We are not yet able to show that the tangent flow at infinity is independent of the scaling factors.

\begin{Theorem}
\label{thm: rough tangent}
The tangent flow at infinity of an ancient $H$-concentrated (c.f. \cite[Definition 3.30]{Bam20b})  metric flow does not depend on the base point.
\end{Theorem}

Under the assumptions of Theorem \ref{Thm_main} or Theorem \ref{Thm_main_1}, the limit in (\ref{the_F_convergence}) is smooth, and hence the $\mathbb{F}$-convergence is automatically smooth convergence according to \cite[Theorem 2.5]{Bam20c}. Note that the scaling in (\ref{the_F_convergence}) is a Type I scaling. We are indeed able to further weaken this condition by a noncollapsing argument of Bamler \cite[Theorem 6.2]{Bam20a}. More precisely, we show that any Ricci flow admitting a closed forward/backward limit must be a forward/backward Type I Ricci flow, and hence satisfies all the results established thus far. In particular, we \emph{do not} assume that the scaled limit is obtained from a Type I scaling process. The following two theorems are stronger versions of all the results stated above, hence we will be focusing on them in most part of this paper.

\begin{Theorem}[Rigidity of closed backward limit]\label{Thm_backwardrigidity}
Let $(M^n,g_t)_{t\in (-\infty,0]}$ be an ancient Ricci flow satisfying (\ref{Nash_bound}), where $M$ is a closed manifold. Assume that there exist $t_i\searrow-\infty$ and $Q_i>0$ such that
\begin{eqnarray}\label{C2convergence0}
\left(M^n,Q_ig_{t_i}\right)\longrightarrow(M^n,g_o),
\end{eqnarray}
in the $C^2$ Cheeger-Gromov sense, where $g_o$ is a $C^2$ Riemannian metric on $M^n$. Then $g_o$ is a shrinker metric and $(M,g_t)$ is a Type I ancient solution satisfying Theorem \ref{Thm_main} and Corollary \ref{Coro_main}
with the unique tangent flow at infinity $\left(\XX^\infty,(\mu^\infty_t)_{t\in -(\infty,0)}\right)$ being the canonical form of $(M,g_o,f_o)$, where $f_o$ is the potential function of $g_o$ and $\mu^\infty_t$ is the conjugate heat flow constructed using $f_o$.
\end{Theorem}

\begin{Theorem}[Rigidity of closed forward limit]\label{Thm_forwardrigidity}
Let $(M^n,g_t)_{t\in [-T,0)}$ be a Ricci flow, where $M$ is a closed manifold. Assume that there exist $t_i\nearrow 0$ and $Q_i\nearrow \infty$ such that
\begin{eqnarray}\label{C2convergence}
\left(M^n,Q_ig_{t_i}\right)\longrightarrow(M^n,g_o),
\end{eqnarray}
in the $C^2$ Cheeger-Gromov sense, where $g_o$ is a $C^2$ Riemannian metric on $M^n$. Then $g_o$ is a shrinker metric and $(M,g_t)$ forms a Type I singularity at $t=0$ and satisfies Theorem \ref{Thm_main_1} and Corollary \ref{Coro_main_1}
with the unique tangent flow $\left(\XX^\infty,(\mu^\infty_t)_{t\in -(\infty,0)}\right)$ at $t=0$ being the canonical form of $(M,g_o,f_o)$, where $f_o$ is the potential function of $g_o$ and $\mu^\infty_t$ is the conjugate heat flow constructed using $f_o$.
\end{Theorem}

The technique implemented in proving Theorem \ref{Thm_backwardrigidity} also implies the following interesting dichotomy for ancient solutions admitting a closed backward limit. This theorem supplements the case where (\ref{Nash_bound}) does not hold. Before stating the theorem, let us recall the definition of strong noncollapseness.
A Riemannian manifold $(M,g)$ is \emph{strongly} $\kappa$-\emph{noncollapsed} for some constant $\kappa>0,$ if
\[
    {\rm Vol}(B(x,r)) \ge \kappa r^n,\quad
    \text{ whenever } \sup_{B(x,r)}R \le r^{-2}.
\]
A Ricci flow $(M^n,g_t)_{t\in I}$ is called strongly noncollapsed if there is a constant $\kappa>0,$ such that each time slice $(M,g_t)$ is strongly $\kappa$-noncollapsed.

\begin{Theorem}[A dichotomy for ancient Ricci flows with closed backward limits]\label{Thm_dichotomy}
Let $(M^n,g_t)_{t\in (-\infty,0]}$ be an ancient Ricci flow, where $M$ is a closed manifold. Assume that there exist $t_i\searrow-\infty$ and $Q_i>0$ such that
\begin{eqnarray}\label{C2convergence2}
\left(M^n,Q_ig_{t_i}\right)\longrightarrow(M^n,g_o),
\end{eqnarray}
in the $C^2$ Cheeger-Gromov sense, where $g_o$ is a $C^2$ Riemannian metric on $M^n$. Then \emph{either one} of the following is true.
\begin{enumerate}[(1)]
    \item $Q_i|t_i|$ is bounded: $(M,g_t)$ is a Type I and noncollapsed ancient Ricci flow, satisfying (\ref{Nash_bound}) and all conclusions of Theorem \ref{Thm_backwardrigidity}.
    \item  $Q_i|t_i|\to \infty$: $(M,g_t)$ is a Type II collapsed ancient Ricci flow, (\ref{Nash_bound}) fails, and $g_o$ is a Ricci flat metric. Here by collapsed we mean not noncollapsed in the strong sense.
\end{enumerate}
\end{Theorem}

Next, we shall present one consequence of our main results. Let $\left(M^n, g_t\right)_{t \in [-T, 0)}$ be a Ricci flow on a closed manifold $M^n$ which develops a finite time singularity at $t=0$. A complete ancient solution $\left(N^n,g_{\infty,t}\right)_{t\in(-\infty,0]}$ is called a \emph{singularity model} of $\left(M^n, g_t\right)_{t \in [-T, 0)}$, if there exist $t_i\nearrow 0$, $Q_i\nearrow\infty$, and $p_i\in M$, such that 
\begin{equation}\label{convergencetosingularitymodel}
    \left(\left(M^n,Q_ig_{t_i+Q_i^{-1}t}\right)_{t\in[-Q_i(T+t_i),0]},p_i\right)\longrightarrow\left((N^n,g_{\infty,t})_{t\in(-\infty,0]},p_\infty\right)
\end{equation}
in the smooth Cheeger-Gromov-Hamilton sense, where $p_\infty$ is a point on $N^n$. If $N^n$ is a closed manifold, then obviously $N^n$ is diffeomorphic to $M^n$, and Z. Zhang \cite{Zh07} proved that in this case $\left(M^n,g_{\infty,t}\right)_{t\in(-\infty,0]}$ must be (the canonical form of) a \emph{shrinking gradient Ricci soliton} by showing that this ancient solution has constant $\nu$-functional. The following question is proposed in \cite{RFV3}:
\begin{quotation}
PROBLEM 17.17. Show that if $(\mathcal{M}^n, g(t))$, $t\in [0,T)$, is a finite time singular solution forming a singularity model on a closed manifold (which then must be diffeomorphic to $\mathcal{M}$), then $g(t)$ is Type I.
\end{quotation}

And we shall give the affirmative answer to the above question in the following corollary of Theorem \ref{Thm_forwardrigidity}.

\begin{Corollary}
A Ricci flow admitting a closed finite-time singularity model must have a Type I singularity.
\end{Corollary}

This paper is organized as follows. In Section 2 we review some basic notions and introduce some important techniques. In Section 3 we show that under the assumptions of Theorem \ref{Thm_backwardrigidity} or Theorem \ref{Thm_forwardrigidity}, the scaling factors must be of Type I. In Section 4, we study the geometric properties of the sequential limit $(M^n,g_o)$ arising from Theorem \ref{Thm_backwardrigidity}, Theorem \ref{Thm_forwardrigidity}, and Theorem \ref{Thm_dichotomy}, and Theorem \ref{Thm_dichotomy}(2) is proved at the end of this section. In Section 5, we prove Theorem \ref{Thm_backwardrigidity} and Theorem \ref{Thm_forwardrigidity}, and thereby prove Theorem \ref{Thm_main}, Theorem \ref{Thm_main_1}, Corollary \ref{Coro_main}, and Corollary \ref{Coro_main_1}. In Section 6, we prove Theorem \ref{thm: rough tangent}.

\textbf{Acknowledgements.} The authors would like to thank Professor Richard Bamler for introducing the problem that tangent flows at infinity do not depend on base points, i.e. Theorem \ref{thm: rough tangent}, and for sketching the ideas. P.-Y. Chan was partially supported by an AMS–Simons Travel Grant.

\section{Preliminaries}

\subsection{Perelman's entropy and Nash entropy}
Let $(M^n,g)$ be a closed Riemannian manifold, Perelman \cite{Per02} defined the following functionals:
\begin{eqnarray*}
\W(g,f,\tau)&:=&\int_M\big(\tau(|\nabla f|^2+R)+f-n\big)(4\pi\tau)^{-\frac{n}{2}}e^{-f}dg,\quad \tau>0,
\\
\mu(g,\tau)&:=& \inf\left\{\W(g,f,\tau)\, \bigg\vert\, \int_{M}(4\pi\tau)^{-\frac{n}{2}}e^{-f}dg=1\right\},
\\
\nu(g,\tau)&:=&\inf_{0<s\leq\tau}\mu(g,\tau),\quad \nu(g)\ :=\ \inf_{\tau>0}\mu(g,\tau).
\end{eqnarray*}
It is well-known that $\mu(g,\tau)$ is the logarithmic Sobolev constant at scale $\tau$ and $\nu(g)$ is the Sobolev constant.

Let $(M^n,g_t)_{t\in[-T,0)}$ be a Ricci flow, where $M$ is a closed manifold. Let $(p_0,t_0)\in M^n\times(-T,0)$ be a fixed point in the space-time and let $$v(\cdot,t):=(4\pi(t_0-t))^{-\frac{n}{2}}e^{-f(\cdot,t)}$$ be the conjugate heat kernel based at $(p_0,t_0)$. We shall follow Bamler \cite{Bam20a} and call the evolving probability measure
$$\nu_{p_0,t_0\,|\, t}(A):=\int_{A}v(\cdot,t)dg_t\quad\text{ for all measurable }\quad A\subset M$$
the \emph{conjugate heat kernel} based at $(p_0,t_0)$ as well. More generally, one may also use an arbitrary positive solution to the conjugate heat equation with unit integral to construct an evolving probability measure, and we shall call such an evolving probability measure a \emph{conjugate heat flow}. Then Perelman's entropy and the Nash entropy are respectively defined as
\begin{eqnarray*}
\W_{p_0,t_0}(\tau)&:=&\W(g_{t_0-\tau},f(\cdot,t_0-\tau),\tau),
\\
\N_{p_0,t_0}(\tau)&:=&\int_Mf(\cdot,t_0-\tau)\,d\nu_{p_0,t_0\,|\,t_0-\tau}-\frac{n}{2}.
\end{eqnarray*}
Perelman's monotonicity formula implies that both $\W_{p_0,t_0}(\tau)$ and $\N_{p_0,t_0}(\tau)$ are decreasing in $\tau$. Indeed, we have
\begin{eqnarray}\label{monotonicityofentropy}
\frac{d}{d\tau} \W_{p_0,t_0}(\tau)&=&-2\int_M\tau\left|\Ric+\nabla^2f-\frac{1}{2\tau}g\right|^2(\cdot,t_0-\tau)\,d\nu_{p_0,t_0\,|\,t_0-\tau},
\\\nonumber
\W_{p_0,t_0}(\tau)&\leq& \N_{p_0,t_0}(\tau)\leq 0\quad\text{ for all }\quad \tau>0.
\end{eqnarray}

If $t=0$ is a singular time of the Ricci flow, by \cite{MM15}, we may let $t_i\nearrow 0$ and consider a sequence of conjugate heat kernels based at $(p_0,t_i)$. If $M^n$ is closed, then, after passing to a subsequence, we obtain a ``singular'' conjugate heat kernel $$v(\cdot,t):=(4\pi|t|)^{-\frac{n}{2}}e^{-f(\cdot,t)},\quad d\nu_{p_0,0\,|\,t}:=v(\cdot,t)dg_t.$$
Obviously, $\nu_{p_0,0\,|\,t}$ is still a probability measure for each $t\in[-T,0)$. Hence, we may use this singular conjugate heat kernel to define $\W_{p_0,0}(\tau)$ and $\N_{p_0,0}(\tau)$ as well, and one may easily verify that (\ref{monotonicityofentropy}) still holds in this case. Note that the singular conjugate heat kernel based at a fixed point is not necessarily unique.

\subsection{Forward and backward modified Ricci flow}

For any closed Riemannian manifold $(M^n,g)$, we define
\begin{eqnarray*}
\mu_g&:=&\mu(g,1)=\inf\left\{\W(g,f,1)\, \bigg\vert\, \int_{M}(4\pi)^{-\frac{n}{2}}e^{-f}dg=1\right\}.
\end{eqnarray*}
Although the minimizer of $\W(g,\cdot,1)$ always exists, it is not necessarily unique (see \cite[Lemma 17.22]{RFV3}). If the minimizer is unique, we shall denote it by $f_g$. By a straightforward computation, the $L^2$ gradient of the $\mu_{(\cdot)}$ functional is
$$\nabla \mu_g=-2\left(\Ric_{g}+\nabla^2f_g-\frac{1}{2}g\right).$$
The gradient flow of $\mu_g$
$$\frac{\partial}{\partial t}g_t=-2\left(\Ric_{g_t}+\nabla^2f_{g_t}-\frac{1}{2}g_t\right)$$
is called the \emph{modified Ricci flow}. We shall also define the \emph{backward modified Ricci flow} as
$$\frac{\partial}{\partial t}g_t=2\left(\Ric_{g_t}+\nabla^2f_{g_t}-\frac{1}{2}g_t\right)$$

In fact, given arbitrary initial data, the modified Ricci flow does not always exist, but when it exists, it is the composite of a Ricci flow with a one-parameter family of diffeomorphism and a time-scaling. Obviously, along the forward/backward modified Ricci flow, $\mu_{g(t)}$ is monotonically increasing/decreasing.

\subsection{Regular neighborhood and \L ojaciewicz inequality}

In this subsection we introduce some important techniques introduced by Sun-Wang \cite{SW15}, especially the \L ojaciewicz inequality for the Ricci flow. Let $(M^n,g_o,f_o)$ be a compact Ricci shrinker, normalized in the way that 
\begin{gather}\label{shrinker_normalization}
    \Ric_{g_o}+\nabla^2 f_o=\frac{1}{2}g_o,
    \\\nonumber
    \int_M (4\pi)^{-\frac{n}{2}}e^{-f_o}dg_o=1.
\end{gather}
Then $f_o\equiv f_{g_o}$ is the unique minimizer of $\W(g_o,\cdot\,,1)$. Letting $\R(M)$ be the space of Riemannian metrics on $M$, $k\,\in \mathbb{N}$, $\gamma\in (0,1)$, the $C^{k,\gamma}$ neighborhood of $g_o$ is defined to be
\begin{equation*}
    \mathcal{V}^{k,\gamma}_{\delta}:=\left\{g\in \R(M)\,\Big|\, \|g-g_o\|_{C^{k,\gamma}_{g_o}}\leq \delta\right\}.
\end{equation*}
Sung-Wang \cite{SW15} proved the following \L ojaciewicz inequality:
\begin{Theorem}[Lemma 3.1 in \cite{SW15}]\label{Lojaciewicz}
There is a $C^{k,\gamma}$ ($k\gg 1$) neighborhood $\mathcal U$ of $g_o$, called a \emph{regular neighborhood}, such that for all $g\in\mathcal{U}$, there is a unique $f_g$, and the map $P:f\mapsto f_g$ is analytic on $\mathcal U$. Furthermore, there are constants $C>0$ and $\alpha\in[\frac{1}{2},1)$, such that for any $g\in\mathcal{U}$, we have
$$\|\nabla \mu_g\|_{L^2_g}\geq C|\mu_{g_o}-\mu_g|^\alpha,$$ where $\|\,\cdot\,\|_{L^2_g}$ is the $L^2$ norm taken with respect to $g$. 
\end{Theorem}

The following theorem is a consequence of the above \L ojaciewic inequality.

\begin{Theorem}[Lemma 3.2 in \cite{SW15}]\label{neighborhoods}
Suppose $(M^n,g_o,f_o)$ is a closed Ricci shrinker normalized as in \eqref{shrinker_normalization}.
Let $(M^n,g_t)_{t\in[0,T)}$ be a modified Ricci flow
\[
    \frac{\partial}{\partial t}g_t  =  -2\left(\Ric_{g_t}+\nabla^2f_{g_t}-\tfrac{1}{2}g_t\right),
\]
where $T\in(0,\infty]$ and $[0,T)$ is the maximum interval of existence for $g_t$. Furthermore, assume $\mu_{g_t}\leq \mu_{g_o}$ for all $t\in[0,T)$. Then for all $\varepsilon>0$ small enough, there is $\delta\in(0,\varepsilon)$, such that if $g_0\in \V^{k+10,\gamma}_{\delta}\subset\mathcal U$, then
$g_t\in \V^{k,\gamma}_{\varepsilon}\subset\mathcal U$ for all $t\in[0,T)$. If this ever happens, and if, in addition, $T=\infty$, then we also have that $g_t$ converges in the $C^{k,\gamma}_{g_o}$ norm to a limit $g_\infty$, which is also a shrinker metric in $\mathcal U$ with $\mu_{g_o}=\mu_{g_\infty}$. The convergence is in a polynomial rate, that is
$$\left\|g_\infty-g_t\right\|_{C^{k,\gamma}_{g_o}}\leq C t^{-\beta}\quad\text{ for all }\quad t\geq 1,$$
where $\beta$ is a constant depending only on $\alpha$ in the statement of Theorem \ref{Lojaciewicz}.
\end{Theorem}

Note that the above theorem is slightly weaker than the original statement of \cite[Theorem 3.2, Theorem 3.3]{SW15}, since we have stated it in the way corresponding to the backward version below, which can be proved by using the same technique, i.e. the \L ojasiewicz argument. The proof of Theorem \ref{thm:backward convergence} is not essentially different from that of Theorem \ref{neighborhoods}. Nevertheless, for the convenience of the reader, we have included its proof in the appendix. Note that in Theorem \ref{neighborhoods}, it is necessary to assume $g_0\in\mathcal{V}^{k+10,\gamma}_{\delta}$ in stead of $g_0\in\mathcal V^{k,\gamma}_{\delta}$, since the forward short-time stability (i.e. $g_t\in\mathcal V^{k,\gamma}_{\varepsilon}$ for a short time) requires higher regularity of $g_0$.

\begin{Theorem}[Backward version]\label{thm:backward convergence}
Suppose that $(M^n,g_o,f_o)$ is a closed Ricci shrinker normalized as in \eqref{shrinker_normalization}.
Let $(M^n,g_t)_{t\in[0,T)}$ be a backward modified Ricci flow
\[
    \frac{\partial}{\partial t}g_t  =  2\left(\Ric_{g_t}+\nabla^2f_{g_t}-\tfrac{1}{2}g_t\right),
\]
where $T\in(0,\infty]$ and $[0,T)$ is the maximum interval of existence for $g_t$. Furthermore, assume $\mu_{g_t}\geq \mu_{g_o}$ for all $t\in[0,T)$. Then for all $\varepsilon>0$ small enough, there is $\delta\in(0,\varepsilon)$, such that if $g_0\in \V^{k,\gamma}_{\delta}\subset\mathcal U$, then
$g_t\in \V^{k,\gamma}_{\varepsilon}\subset\mathcal U$ for all $t\in[0,T-\varepsilon)$. If this ever happens, and if, in addition, $T=\infty$ (in which case we set $T-\varepsilon=\infty$ by default), then we also have that $g_t$ converges in $C^{k,\gamma}_{g_o}$ norm to a limit $g_\infty$ which is also a shrinker metric in $\mathcal U$  with $\mu_{g_o}=\mu_{g_\infty}.$ The convergence is in a polynomial rate, that is
$$\left\|g_\infty-g_t\right\|_{C^{k,\gamma}_{g_o}}\leq C t^{-\beta}\quad\text{ for all }\quad t\geq 1,$$
where $\beta$ is a constant depending only on $\alpha$ in the statement of Theorem \ref{Lojaciewicz}.
\end{Theorem}

\subsection{A noncollapsing estimate of Bamler}

The following theorem of Bamler shows that for a Ricci flow with bounded Nash entropy, the volume of a disk around an $H_n$-center cannot become too collapsed. Let $(M^n,g_t)_{t\in I}$ be a Ricci flow, where $M$ is a closed manifold. Let $(\nu_{x_0,t_0\,|\,t})_{t<t_0}$ be the conjugate heat kernel based at $(x_0,t_0)$.
An $H_n$-center $(z,t)$ of $(x_0,t_0)$, where $t<t_0$, can be understood as the point around which $\nu_{x_0,t_0\,|\, t}$ accumulates its mass. An important fact we shall used concerning an $H_n$-center is that \emph{it always exists}; see \cite{Bam20a} for more details.

\begin{Theorem}[Theorem 6.2 in \cite{Bam20a}]\label{concollapsing}
Let $(M^n,g_t)_{t\in I}$ be a Ricci flow on a closed manifold. Let $t\in I$ and $r>0$ such that $[t-r^2,t]\subset I$. Assume that $(z,t-r^2)\in M\times I$ is an $H_n$-center of some $(x,t)\in M\times I$. If $R_{g_{t-r^2}}\geq R_{\operatorname{min}}$, then we have
$$\operatorname{Vol}_{g_{t-r^2}}\left(B_{g_{t-r^2}}\left(z,\sqrt{2H_n}r\right)\right)\geq c(R_{\operatorname{min}}r^2)\exp(\mathcal{N}_{x,t}(r^2))r^n,$$ where $c(R_{\operatorname{min}}r^2)=c(n)\exp(-2(n-2R_{\operatorname{min}}r^2)^{\frac{1}{2}})$ and $H_n:=\frac{(n-1)\pi^2}{2}+4$.
\end{Theorem}

\subsection{Bamler's Tangent flows}

\def \IR {\mathbb{R}}
\def \YY {\mathcal{Y}}
\def \TT {\mathcal{T}}

When introducing the notion of tangent flow, we implement the same notations as in \cite{Bam20b}, to which the reader is encouraged to refer for more details. For a metric flow $\XX$ over $I\subset \IR,$ we denote by
$\XX^{-t_0,\lambda}$ the metric flow obtained by first applying a $-t_0$ time shift to $\XX$ and then a parabolic rescaling by factor $\lambda$ (as the notation in (\ref{the_F_convergence})). Let $\XX$ be a metric flow over
$I$ and $|(-\infty,0]\setminus I|=0.$
For any $x_0\in \XX_{t_0},$
we call a metric flow pair (that is, a metric flow coupled with a conjugate heat flow; see \cite[Section 5]{Bam20b}) 
$(\XX^\infty, (\nu^\infty_{x_{\max}\,|\,t})_{t\in I'^{,\infty}})$ a \textbf{tangent flow at infinity based at}  $x_0$ if
there is a sequence $\lambda_j\searrow 0$, such that, 
\[
\left(\XX^{-t_0,\lambda_j}_{[-T,0]},
\left(\nu_{x_0\,|\,t}^{-t_0,\lambda_j}\right)_{t\in \lambda_j^2(I-t_0)\cap [-T,0]}\right)
\xrightarrow{\makebox[1cm]{$\mathbb{F}$}}
\left(\XX^\infty_{[-T,0]}, \left(\nu^\infty_{x_{\max}\,|\,t}\right)_{t\in  I'^{,\infty}\cap[-T,0]}\right).
\]
for any $T<\infty$. We can then define
\[
    \TT_{x_0}^\infty
    :=\{\text{ tangent flows at infinity based at }x_0\},
\]
which is nonempty by Bamler's compactness theory in \cite[Section 7]{Bam20b}. Now we may restate Theorem \ref{thm: rough tangent} more precisely as follows. Note that in the statement of Theorem \ref{thm: tangent flows}, the base points $x_0$ and $y_0$ need not lie in the same time-slice.
\begin{Theorem} \label{thm: tangent flows}
Suppose that $\XX$ is an $H$-concentrated (\cite[Definition 3.30]{Bam20b}) 
metric flow over $(-\infty,0]$.
Then for any $x_0$, $y_0\in \XX,$ we have, 
up to isometry, 
\[
    \TT_{x_0}^\infty=\TT_{y_0}^\infty.
\]
\end{Theorem}

Similarly, recall that $(\XX^\infty, (\nu^\infty_{x_{\max}\,|\,t})_{t\in I'^{,\infty}})$ is said to be a \textbf{tangent flow at}  $x_0$ if
there is a sequence $\lambda_j\to \infty$, such that
\[
\left(\XX^{-t_0,\lambda_j}_{[-T,0]},
\left(\nu_{x_0\,|\,t}^{-t_0,\lambda_j}\right)_{t\in \lambda_j^2(I-t_0)\cap [-T,0]}\right)
\xrightarrow{\makebox[1cm]{$\mathbb{F}$}}
\left(\XX^\infty_{[-T,0]}, \left(\nu^\infty_{x_{\max}\,|\,t}\right)_{t\in  I'^{,\infty}\cap[-T,0]}\right)
\]
for any $T<\infty.$ Note that if $\XX$ is generated by a Ricci flow forming a finite-time singularity, then $x_0$ is allowed to be in the ``singular time-slice'', and if this is the case, then $\nu_{x_0\,|\,t}$ is understood to be a singular conjugate heat kernel as described in Section 2.1.

\subsection{Forward and backward pseudolocality}

Perelman first proved the forward pseudolocality theorem \cite[Section 10]{Per02}, showing that a relatively regular region cannot become singular under Ricci flow too quickly. We shall use the following version proved by Lu \cite{Lu10}.

\begin{Theorem}[\cite{Per02,Lu10}]
\label{thm: forward pseudo}
Fix any $n\ge 2$ and $v>0,$ there is $\epsilon_0=\epsilon_0(n,v)>0$ with the following property.
For any $r>0$ and $\epsilon\in (0,\epsilon_0],$ suppose that $(M^n,g_t)_{t\in [0,(\epsilon r)^2]}$ is a complete Ricci flow with bounded curvature. 
Suppose that for some  $x_0\in M,$ we have
\[
    \sup_{B_{g_0}(x_0,r)\times\{0\}} |{\Rm}|\le r^{-2},\quad
    {\rm Vol}_{g_0}\left(B_{g_0}(x_0,r)\right) \ge v r^n.
\]
Then
\[
    |{\Rm}|(x,t)
    \le (\epsilon_0 r)^{-2},
\]
for any $t\in [0, (\epsilon_0 r)^2]$ and $x\in B_{g_t}(x_0,\epsilon_0 r).$ 
\end{Theorem}

We will also use a backward pseudolocality theorem recently proved by Bamler \cite[Theorem 2.47]{Bam20c}. Although the following theorem bears a similar name as the above theorem, yet their proofs are substantially different.

\begin{Theorem}[\cite{Bam20c}]
\label{thm: backward pseudo}
For any $n\ge 2$ and $\alpha>0,$ there is an $\epsilon(n,\alpha)>0,$ such that the following holds.
Let $(M^n,g_t)_{t\in [-r^2,0]}$ be a compact Ricci flow for some $r>0$. Suppose that for some point $x_0\in M$, we have
\[
    \sup_{B_{g_0}(x_0,r)\times\{0\}} |{\Rm}|
    \le (\alpha r)^{-2},\quad
    {\rm Vol}_{g_0}\left(
    B_{g_0}(x_0,r)
    \right)\ge \alpha r^n.
\]
Then
\[
    |{\Rm}|\le (\epsilon r)^{-2},
\]
on $B_{g_0}(x_0,(1-\alpha)r)\times [-(\epsilon r)^2,0].$

\end{Theorem}

\section{Type I bounds of the scaling factors}

In this section, we shall show that under the conditions of Theorem \ref{Thm_backwardrigidity} or Theorem \ref{Thm_forwardrigidity}, the scaling factors must have the following Type I bounds.

\begin{Proposition}\label{typeI}
Under the assumptions of either Theorem \ref{Thm_backwardrigidity} or Theorem \ref{Thm_forwardrigidity}, we have
\begin{equation}\label{eq:typeI}
    c\leq |t_i|Q_i\leq C\quad\text{ for all }\quad i\in\mathbb{N},
\end{equation}
where $0<c<C<\infty$ are constants independent of $i\in\mathbb{N}$.
\end{Proposition}

\begin{proof}
Let us first of all consider the case of Theorem \ref{Thm_forwardrigidity}, and the case of Theorem \ref{Thm_backwardrigidity} is similar.

The lower bound is straightforward. By the $C^2$ convergence (\ref{C2convergence}), we have 
$$Q_i^{-1}\max_M\left|{\Rm_{g_{t_i}}}\right|=\max_{M}\left|\Rm_{Q_ig_{t_i}}\right|\longrightarrow \max_M|{\Rm_{g_o}}|.$$
On the other hand, by applying the maximum principle to the inequality 
$$\frac{\partial}{\partial t}|{\Rm}|-\Delta |{\Rm}|\leq 8|{\Rm}|^2,$$
we have 
\begin{eqnarray}\label{nonsenseRmmp}
\left(\max_M|{\Rm_{g_t}}|\right)^{-1}\leq 
c_0^{-1} + 8|t|
\quad\text{ for all }\quad t<0,
\end{eqnarray}
where $\displaystyle c_0:=\limsup_{t\to 0}\max_M|\Rm_{g_t}|\in(0,\infty]$. 
Hence, in the case of Theorem \ref{Thm_forwardrigidity}, we have $|t|\max_M|{\Rm_{g_t}}|\geq \frac{1}{8}$ for all $t\in[-T,0)$, and
$$\limsup_{i\to\infty}(Q_i|t_i|)^{-1}\leq 8\limsup_{i\to\infty} Q_i^{-1}\max_M\left|\Rm_{g_{t_i}}\right|=8\max_M|{\Rm_{g_o}}|<\infty.$$

We shall then prove the upper bound in (\ref{eq:typeI}) by contradiction. By possibly passing to a subsequence, we assume that
\begin{equation}\label{contradictionassuption}
    \lim_{i\to\infty} Q_i|t_i|=\infty.
\end{equation}
By the $C^2$ convergence (\ref{C2convergence}) again, we have
\begin{equation}\label{vol con}
    \lim_{i\to\infty}\operatorname{Vol}(M,Q_ig_{t_i})=\operatorname{Vol}(M,g_o)\in (0,\infty). 
\end{equation}
Hence we have
\begin{eqnarray*}
\lim_{i\to\infty} \operatorname{Vol}(M,|t_i|^{-1}g_{t_i})&=&\lim_{i\to\infty} \operatorname{Vol}\Big(M,(Q_i|t_i|)^{-1}Q_ig_{t_i}\Big)
\\
&=&\lim_{i\to\infty}\frac{\operatorname{Vol}(M,Q_ig_{t_i})}{(Q_i|t_i|)^{\frac{n}{2}}}
\\
&=&0,
\end{eqnarray*}
where we have applied both (\ref{contradictionassuption}) and (\ref{vol con}). This is equavalent to
\begin{equation}\label{contra}
    \lim_{i\to\infty}\frac{\operatorname{Vol}(M,g_{t_i})}{|t_i|^{\frac{n}{2}}}=0.
\end{equation}

Next, we fix a point $p_0\in M$. For any $i\in\mathbb{N}$, let $(p_i,t_i)\in M\times[-T,0)$ be an $H_n$-center of $(p_0,\frac{t_i}{2})$, and let $r^2=\frac{1}{2}|t_i|$. By Perelman's monotonicity formula, we have
\begin{eqnarray*}
\N_{p_0,\frac{t_i}{2}}(r^2)\geq \W_{p_0,\frac{t_i}{2}}(r^2)\geq \mu(g_{t_i},r^2)\geq\mu\big(g_{-T},r^2+t_i+T\big)\geq \nu(g_{-T},2T)>-\infty.
\end{eqnarray*}
Note that $\nu(g_{-T},2T)$ is a constant depending only on the geometry of $(M^n,g_{-T})$. Since $t_i\nearrow 0$, we may, without loss of generality, assume that $t_i\geq-\frac{T}{2}$ for all $i\in\mathbb{N}$. The maximum principle implies that 
$$R_{g_{t_i}}\geq \inf_M R_{g_{-T/2}}\geq-\tfrac{n}{T}.$$
We may then apply Theorem \ref{concollapsing} with $(x,t)=(p_0,\frac{t_i}{2})$, $r=\sqrt{\frac{|t_i|}{2}}\leq\sqrt{\frac{T}{2}}$, $(z,t-r^2)=(p_i,t_i-r^2)$, and $R_{\operatorname{min}}=-\frac{n}{T}$. This yields
\begin{eqnarray*}
\operatorname{Vol}_{g_{t_i}}\left(B_{g_{t_i}}(p_i,\sqrt{2H_n}r)\right)&\geq& c(n)\exp\left(-2(n-2R_{\operatorname{min}}r^2)^{\frac{1}{2}}\right)\cdot\exp\left(\mathcal{N}_{p_0,\frac{t_i}{2}}(r^2)\right)\cdot r^n
\\
&\geq& c(n)\exp\left(-2\left(n+\tfrac{n}{T}\cdot\left(\sqrt{\tfrac{T}{2}}\right)^2\right)^{\frac{1}{2}}\right)\cdot\exp\Big(\nu(g(-T),2T)\Big)\cdot r^n
\\
&\geq& c_0|t_i|^{\frac{n}{2}},
\end{eqnarray*}
where $c_0$ is a constant independent of $i\in\mathbb{N}$; this is clearly a contradiction against (\ref{contra}) and the proposition follows immediately.

Let us now consider the case of Theorem \ref{Thm_backwardrigidity}. By (\ref{nonsenseRmmp}) again, we also have the lower bound in (\ref{eq:typeI}). As to the upper bound, we apply \cite[Proposition 4.6]{MZ21} together with the condition (\ref{Nash_bound}) and obtain 
$$\N_{p_0,\frac{t_i}{2}}(\tau)\geq -Y\quad\text{ for all }\quad \tau>0,$$
where $Y$ is the constant in (\ref{Nash_bound}) and is, in particular, independent of $i$. Then the same argument above apparently applies to the current case. Note that in this case we can take $R_{\operatorname{min}}=0$.
\end{proof}

The above proposition shows that in the ancient case, bounded Nash entropy implies the Type I bound of the scaling factors. In order to obtain the exact dichotomy of Theorem \ref{Thm_dichotomy}, we shall prove that the reverse is also true.

\begin{Proposition}\label{reverse_typeI}
Under the assumption of Theorem \ref{Thm_dichotomy} (that is, the $C^2$ convergence in (\ref{C2convergence2})), if $Q_i|t_i|$ is bounded from above, then (\ref{Nash_bound}) holds.
\end{Proposition}
\begin{proof}
Since $Q_i|t_i|$ is bounded from above, by (\ref{vol con}) we have
\begin{eqnarray}\label{thevolumenonsense}
\Vol_{g_{t_i}}(M,g_{t_i})\geq c|t_i|^{\frac{n}{2}},
\end{eqnarray}
where $c$ is a constant independent of $i$. On the other hand, since, by (\ref{C2convergence2}),
$$\operatorname{diam}(M,Q_ig_{t_i})\longrightarrow\operatorname{diam}(M,g_o)<\infty,$$
we also have 
\begin{eqnarray}\label{thediameternonsense}
\operatorname{diam}(M,g_{t_i})\leq CQ_i^{-\frac{1}{2}}\leq C\sqrt{|t_i|},
\end{eqnarray}
where $C$ is a constant independent of $i$. Note that we have applied the lower bound in (\ref{eq:typeI}), which is automatically true. Fixing an arbitrary point $p_0\in M$, by \cite[Theorem 1.8]{CMZ21} and its remarks we have
$$\N_{p_0,0}(|t_i|)\geq-Y\quad\text{ for all }\quad i,$$
where $Y$ is a positive constant independent of $i$. To see this, note that for an $H_n$-center $(z,t_i)$ of $(p_0,0)$, the disk $B_{g_{t_i}}(z,C\sqrt{|t_i|})$ contains the whole manifold by (\ref{thediameternonsense}) and hence has a volume lower bound by (\ref{thevolumenonsense}). Since $|t_i|\nearrow\infty$, the proposition follows from the monotonicity of the Nash entropy.
\end{proof}

\section{Geometric structure of the sequential limit}

In this section we study the geometric properties of the limit metric $g_o$ arising from (\ref{C2convergence0}), (\ref{C2convergence}), and (\ref{C2convergence2}). We shall prove that in the cases of Theorem \ref{Thm_backwardrigidity}, Theorem \ref{Thm_forwardrigidity}, and Theorem \ref{Thm_dichotomy}(1), $g_o$ is a shrinker metric, and in the case of Theorem \ref{Thm_dichotomy}(2), $g_o$ is a Ricci flat metric. In particular, we shall give a proof of Theorem \ref{Thm_dichotomy}(2).

\begin{Lemma}\label{space-time-extension}
Under the assumption of Theorem \ref{Thm_backwardrigidity}, Theorem \ref{Thm_forwardrigidity}, and Theorem \ref{Thm_dichotomy}, there exists $\epsilon>0$, such that, after passing to a subsequence, 
$$\left(M^n,Q_ig_{t_i+Q_i^{-1}t}\right)_{t\in[-\epsilon,\epsilon]}\longrightarrow (M^n,g_{\infty,t})_{t\in[-\epsilon,\epsilon]},$$
where $g_{\infty,t}$ is a Ricci flow with $g_{\infty,0}=g_o$.
\end{Lemma}
\begin{proof}
This lemma is but a consequence of the forward and backward pseudolocality theorems. Denote by $g_{i,t}$ the scaled Ricci flow $Q_ig_{t_i+Q_i^{-1}t}$. By our assumptions, we have $(M,g_{i,0})\to (M,g_o)$ in the $C^2$ Cheeger-Gromov sense. Hence
\begin{eqnarray*}
    \lim_{i\to\infty}\max_M\left|\Rm_{g_{i,0}}\right|&=&\max_M\left|\Rm_{g_o}\right|<\infty
    \\
    \lim_{i\to\infty}\Vol(M,g_{i,0})&=&\Vol(M,g_o)>0,
    \\
    \lim_{i\to\infty}\operatorname{diam}(M,g_{i,0})&=&\operatorname{diam}(M,g_o)<\infty.
\end{eqnarray*}
Therefore, we may find positive constants $r$ and $C$, such that, for some fixed $x_0\in M$, we have
\begin{gather*}
\operatorname{diam}(M,g_{i,0})\leq r,
\\
    |{\Rm_{g_{i,0}}}|\leq Cr^{-2},
    \\
    \Vol_{g_{i,0}}\left(B_{g_{i,0}}(x_0,r)\right)\geq C^{-1}r^n,
\end{gather*}
for all $i\in\mathbb{N}$.

By Theorem \ref{thm: forward pseudo} and Theorem \ref{thm: backward pseudo}, we can find a number $\epsilon_0$, such that, for all $i\in\mathbb{N}$, we have
$$\left|\Rm_{g_{i,t}}\right|\leq (\epsilon_0 r)^{-2}\quad\text{ for all }\quad t\in[-(\epsilon_0r)^2,(\epsilon_0r)^2].$$
Then, letting $2\epsilon=\epsilon_0r^2$, the conclusion follows from Hamilton's compactness theorem.
\end{proof}

After knowing that $g_o$ is a time-slice of the limit Ricci flow of the scaled sequence $(M,Q_ig_{t_i+Q_i^{-1}t})$, it remains to show that this limit Ricci flow is either the canonical form of a Ricci shrinker, or a static Ricci flat manifold. These two cases correspond to the Type I and Type II dichotomy in Theorem \ref{Thm_dichotomy}. Let us first of all consider the Type I case.

\begin{Proposition}
\label{prop: limit is shrinker}
Under the assumptions of Theorem \ref{Thm_backwardrigidity}, Theorem \ref{Thm_forwardrigidity}, and Theorem \ref{Thm_dichotomy}(1), we have that $g_o$ is a shrinker metric whose canonical form is $g_{\infty,t}$ in the above proposition.
\end{Proposition}
\begin{proof}
This proposition is but an application of Perelman's monotonicity formula. Because of Proposition \ref{typeI}, by passing to a subsequence and rescaling, we may, without loss of generality, assume $Q_i\equiv |t_i|^{-1}$. Note that the lower bound in (\ref{eq:typeI}) is automatically true in the case of Theorem \ref{Thm_dichotomy}(1). We shall first of all consider the ancient case, namely, the case of Theorem \ref{Thm_backwardrigidity} and Theorem \ref{Thm_dichotomy}(1). Let $v(\cdot,t):=(4\pi|t|)^{-\frac{n}{2}}e^{-f(\cdot,t)}$ be the conjugate heat kernel based at a fixed point $(p_0,0)$, and define 
\begin{eqnarray*}
g_{i,t}&:=&|t_i|^{-1}g_{|t_i|t},
\\
f_i(\cdot,t)&:=&f(\cdot,|t_i|t),
\\
v_i(\cdot,t)&:=&(4\pi|t|)^{-\frac{n}{2}}e^{-f_i(\cdot,t)}=|t_i|^{\frac{n}{2}}v(\cdot,|t_i|t).
\end{eqnarray*}
Then, by Perelman's monotonicity formula (\ref{monotonicityofentropy}), we have
\begin{align}\label{facundia_1}
& \W_{p_0,0}\big((1+\epsilon)|t_i|\big)-\W_{p_0,0}\big((1-\epsilon)|t_i|\big)
\\\nonumber
=& -2\int_{-(1+\epsilon)}^{-(1-\epsilon)}\int_M|t|\left|\Ric_{g_{i,t}}+\nabla^2_{g_{i,t}}f_i(\cdot,t)-\frac{1}{2|t|}g_{i,t}\right|^2v_i\,dg_{i,t}dt,
\end{align}
where $\epsilon$ is as provided by Lemma \ref{space-time-extension}. Since $\W_{p_0}(\tau)$ is decreasing in $\tau$ and since $\lim_{\tau\to\infty}\W_{p_0,0}(\tau)=\lim_{\tau\to\infty}\N_{p_0,0}(\tau)\geq -Y$ (\cite[Proposition 4.6]{MZ21}), we have
\begin{equation}\label{facundia_2}
   \lim_{i\to\infty}\Big(\W_{p_0,0}\big((1+\epsilon)|t_i|\big)-\W_{p_0,0}\big((1-\epsilon)|t_i|\big)\Big)=0. 
\end{equation}
On the other hand, Lemma \ref{space-time-extension} and the standard regularity theorem of parabolic PDE provide uniform smooth estimates for $v_i$ and $f_i$ on $M^n\times[-(1+\epsilon/2),-(1-\epsilon/2)]$. 
Hence we have the following smooth convergence
$$v_i\to v_\infty, \quad f_i\to f_\infty \quad\text{ on }\quad M^n\times
[-(1+\epsilon/2),-(1-\epsilon/2)],
$$
where $v_\infty=(4\pi|t|)^{-\frac{n}{2}}e^{-f_\infty}$ is positive and integrates to 1. By (\ref{facundia_1}) and (\ref{facundia_2}), we have
$$\Ric_{g_{\infty,t}}+\nabla^2_{g_{\infty,t}}f_\infty(\cdot,t)=\frac{1}{2|t|}g_{\infty,t}\quad\text{ for all }\quad t\in 
[-(1+\epsilon/2),-(1-\epsilon/2)].
$$
This finishes the proof in the ancient case.

In the case of Theorem \ref{Thm_forwardrigidity}, one needs only to replace $v$ with a singular conjugate heat kernel based at $t=0$ (see Section 2.1), and the rest of the arguments are the same. The details are left to the readers.
\end{proof}

Now we are ready to consider the Type II case in Theorem \ref{Thm_dichotomy}.

\begin{Proposition}\label{Ricciflat}
Under the assumptions of Theorem \ref{Thm_dichotomy}(2), $g_o$ is a Ricci flat metric.
\end{Proposition}

\begin{proof}
Let us consider the Ricci flow $g_{\infty,t}$ given by Lemma \ref{space-time-extension}. Recall that $g_o=g_{\infty,0}$. Since in the ancient case, every scaled Ricci flow $g_{i,t}:=Q_ig_{t_i+Q_i^{-1}t}$ in the statement of Lemma \ref{space-time-extension} has nonnegative scalar curvature (c.f. \cite{CBl09}), we have that $g_{\infty,t}$ also has nonnegative scalar curvature. By the strong maximum principle and the curvature evolution equation
$$\frac{\partial}{\partial t} R_{g_{\infty,t}}=\Delta R_{g_{\infty,t}}+2|{\Ric_{g_{\infty,t}}}|^2,$$
we have that $R_{g_{\infty,t}}$ is either positive everywhere or $0$ everywhere, and in the latter case $g_{\infty,t}$ (and hence $g_o$) must be a static Ricci flat metric; we need only to show that the former case does not happen.

Let us assume that $R_{g_o}>c_0>0$ everywhere. Since $(M^n,g_o)$ is closed, it must also satisfy an $L^2$ Sobolev inequality 
\begin{equation*}
    \left(\int_M|u|^{\frac{2n}{n-2}}dg_o\right)^{\frac{n-2}{n}}\leq A\int_M|\nabla u|^2dg_o+B\int_M u^2dg_o\quad\text{ for all }\quad u\in W^{1,2}(M),
\end{equation*}
where $A$ and $B$ are constants depending on the geometry of $(M,g_o)$.

On the other hand, since $(M^n,Q_ig_{t_i})\longrightarrow(M^n,g_o)$ in the smooth sense, we have that for all $i$ large enough, $R_{Q_ig_{t_i}}>\frac{c_0}{2}>0$ everywhere on $M^n$, and $(M^n,Q_ig_{t_i})$ satisfies a Sobolev inequality with constants $2A$ and $2B$. Hence, by Proposition \ref{lowerbound of nu}, we have
\begin{equation*}
    \nu(g_{t_i})=\nu(Q_ig_{t_i})\geq-C\quad\text{ for all $i$ large enough},
\end{equation*}
where $C=C(n,A,B,c_0)$ is independent of $i$. Then, fixing any $p_0\in M$, we have
\begin{eqnarray*}
\N_{p_0,0}(|t_i|)\geq \W_{p_0,0}(|t_i|)\geq \mu(g_{t_i},|t_i|)\geq\nu(g_{t_i})\geq -C,
\end{eqnarray*}
for all $i$ large enough. Therefore, (\ref{Nash_bound}) holds for $(M,g_t)$ and Proposition \ref{typeI} implies that $Q_i|t_i|$ is bounded from above; this contradicts our assumption of Theorem \ref{Thm_dichotomy}(2).
\end{proof}

\begin{proof}[Proof of Theorem \ref{Thm_dichotomy}]
By Proposition \ref{typeI} and Proposition \ref{reverse_typeI}, we have that under the assumption of Theorem \ref{Thm_dichotomy}, the boundedness of $Q_i|t_i|$ is equivalent to (\ref{Nash_bound}). 

Under the assumption of Theorem \ref{Thm_dichotomy}(2), we have that $(M^n,g_o)$ is a closed Ricci flat manifold, and hence is not strongly noncollapsed. It follows that $(M,g_t)$ is not strongly noncollapsed either, since it sequentially converges to $(M,g_o)$.

Under the assumption of Theorem \ref{Thm_dichotomy}(1), we are in the exactly same setting as Theorem \ref{Thm_backwardrigidity}. Hence the remaining part of proof is reduced to the proof of Theorem \ref{Thm_backwardrigidity}.
\end{proof}

\section{Uniqueness of the Type I scaled limit}

In this section, we prove the uniqueness of the tangent flows at infinity and at the finite singular time, i.e., Theorem \ref{Thm_main} and \ref{Thm_main_1}. We will also estimate the convergence rate of the (backward) modified Ricci flow relative to the H\"older norms and prove Corollary \ref{Coro_main} and Corollary \ref{Coro_main_1}.

\subsection{Uniqueness of tangent flow at infinity}

Before we proceed with our argument, we remind the reader that by  Proposition \ref{prop: limit is shrinker}, $g_o$ in (\ref{C2convergence0}) is already known to be a shrinker metric, and we will let $f_o$ be its potential function. Then, Theorem \ref{Thm_main} and Corollary \ref{Coro_main} are implied by the following proposition.
\begin{Proposition}\label{Pro_13}
Let $(M^n,g_t)_{t\in (-\infty,0]}$ be an ancient solution satisfying the conditions in the statements of Theorem \ref{Thm_backwardrigidity}, then $(M,g_t)$ is a Type I ancient Ricci flow, and Theorem \ref{Thm_main} and Corollary \ref{Coro_main} hold for $(M^n,g_t)$ with $\left(\XX^\infty,(\mu^\infty_t)_{t\in -(\infty,0)}\right)$ being the canonical form of $(M,g_o,f_o)$, where $\mu^\infty_t$ is the conjugate heat flow constructed using $f_o$.
\end{Proposition}

\begin{proof}[Proof of Theorem \ref{Thm_main} and Corollary \ref{Coro_main} assuming Proposition \ref{Pro_13}]
Assume that (\ref{the_F_convergence}) holds with $\XX^\infty$ being a metric flow induced by a smooth and closed Ricci flow. Then, by \cite[Theorem 2.5]{Bam20c}, we have that the $\mathbb{F}$-convergence is smooth everywhere. In particular, we have
\begin{eqnarray}\label{C2convergenceinappearance}
\left(M^n,\tau_i^{-1}g_{-\tau_i}\right)\longrightarrow \XX^\infty_{-1}
\end{eqnarray}
in the smooth Cheeger-Gromov sense. As a consequence, we necessarily have that $\XX^\infty_{-1}$ is a smooth Riemannian manifold with the underlying manifold diffeomorphic to $M^n$.  Therefore, (\ref{C2convergenceinappearance}) is the same as (\ref{C2convergence0}), and Theorem \ref{Thm_main} and Corollary \ref{Coro_main} are reduced to Proposition \ref{Pro_13}.
\end{proof}

The rest of this subsection consists of the proof of Proposition \ref{Pro_13}. By Proposition \ref{typeI}, after scaling and passing to a subsequence, we may, without loss of generality, assume that $Q_i|t_i|\equiv 1$ for all $i\in\mathbb{N}$. Let us consider the following dynamic scaling
\begin{eqnarray}\label{defofnormalizedbackwardflow}
\tilde g_s:=e^{-s}g_{-e^s}\quad\text{ for all }\quad s\in[0,\infty).
\end{eqnarray}
Then
\begin{equation}\label{normalizedbackwardflow}
  \frac{\partial}{\partial s} \tg_s
=  2\left(\Ric_{\tg_s}- \tfrac{1}{2}\tg_s\right).  
\end{equation}

$\tg_s$ is not yet a backward modified Ricci flow, but we can potentially construct a backward modified Ricci flow using $\tg_s$ by pulling it back using a flow generated by $\nabla f_{\tg_s}$, where $f_g$ is the minimizer of $\W(g,\cdot,1)$. However, unless we know a priori that $\tg_s$ is already in a regular neighborhood of $g_o$, we cannot guarantee the existence of such a flow, since $f_g$ is not necessarily unique and the map $P:g\to f_g$ is not necessarily smooth (see Section 2.3).

\begin{Lemma}
Let $\varepsilon$ and $\delta$ be small enough such that Theorem \ref{thm:backward convergence} holds for $g_o$. Let $s_i=\log(-t_i)\nearrow\infty$. Then for each $i$ large enough, there is a self-diffeomorphism $\psi_i:M\to M$, such that
$\psi_i^* \tg_{s_i}\in\mathcal V^{k,\gamma}_{\delta}.$
\end{Lemma}
\begin{proof}
By Proposition \ref{prop: limit is shrinker}, we have that 
\begin{equation}\label{SOMEC-G}
    (M^n,\tg_{s_i})\equiv (M^n,|t_i|^{-1}g_{t_i})\longrightarrow (M,g_o)
\end{equation}
in the smooth Cheeger-Gromov sense, and the lemma follows from the definition of Cheeger-Gromov convergence.
\end{proof}

Let us fix an $i_0$ large enough, such that the above Lemma holds for $i=i_0$. Without loss of generality, we shall assume $s_{i_0}=0$ and $\psi_{i_0}=\operatorname{id}$. Then, arguing in the same way as the proof of Lemma \ref{shorttimestability} by using Theorem \ref{thm: backward pseudo}, we have that $$\tg_s\in\mathcal V^{k,\gamma}_{2\delta}\subset \mathcal U\quad\text{ for all }\quad s\in[0,\eta_0],$$
where $\eta_0$ is a small positive number 
depending on $\delta$ and the geometry of $g_o$, 
and $\mathcal U$ is a regular neighborhood of $g_o$. Since $P:g\to f_g$ is analytic on $\mathcal U$, we have that the flow 
\begin{eqnarray}\label{modifying}
\left\{\begin{array}{rl}
     \bg_s&=\ \psi_s^*\tg_s  \\
     \frac{\partial}{\partial s}\psi_s &=\ \nabla_{\tg_s}f_{\tg_s}\circ  \psi_s,\\
     \psi_0 &=\ \operatorname{id},
\end{array}\right.
\end{eqnarray}
is well-defined at least for all $s\in[0,\eta_0]$. Let $[0,T)$ be the maximum interval of existence for the flow $\bg_s$. Obviously, as long as $\bg$ exists, it is a backward modified Ricci flow.

\begin{Lemma}\label{longtime}
$T=\infty$.
\end{Lemma}
\begin{proof}
First of all, we show that $\mu_{\bg_s}\geq\mu_{g_o}$ for all $s\in[0,T)$. Fixing an arbitrary $s\in[0,T)$ and setting $t=-e^s$, we have
$$\mu_{\bg_s}=\mu_{\tg_s}=\mu(g_t,|t|).$$
On the other hand, by (\ref{SOMEC-G}) we have
$$\lim_{i\to\infty}\mu(g_{t_i},|t_i|)=\lim_{i\to\infty}\mu(|t_i|^{-1}g_{t_i},1)=\mu(g_o,1)=\mu_{g_o}.$$
While by the monotonicity of Perelman's $\mu$ functional, we have $$\mu(g_t,|t|)\geq \lim_{i\to\infty}\mu(g_{t_i},|t_i|)\quad\text{ for all }\quad t<0.$$
This shows that $\mu_{\bg_s}\geq\mu_{g_o}$ for all $s\in[0,T)$.

By contradiction, we assume that $T<\infty$. By Theorem \ref{thm:backward convergence}, we have that $\bg_s\in\mathcal V_{\varepsilon}^{k,\gamma}$ for all $s\in[0,T)$. Here we would like to choose $\varepsilon$ small enough, such that $\mathcal V^{k,\gamma}_{2\varepsilon}\subset\mathcal U$. Arguing in the same way as in the proof of Lemma \ref{shorttimestability} again, we can find a $\eta>0$, independent of $\bar s\in[0,T)$, such that 
$$\psi_{\bar s}^*\tg_s\in\mathcal{V}^{k,\gamma}_{2\varepsilon}\subset\mathcal U\quad\text{ for all }\quad s\in[\bar s,\bar s+\eta].$$
Indeed, this follows from applying the backward Pseudolocality theorem (Theorem \ref{thm: backward pseudo}) to the flow $s\to\psi_{\bar s}\tg_s$, which is a backward normalized Ricci flow satisfying (\ref{normalizedbackwardflow}) with the initial data $\psi_{\bar s}\tg_{\bar s}=\bg_{\bar s}\in \mathcal V^{k,\gamma}_{\varepsilon}$. Furthermore, since $\overline g_{\bar s}$ is sufficiently close to $g_o$, we can estimate all of its geometric quantities in terms of $g_o$, and this is why $\eta$ is independent of $\bar s.$ We then have that
\begin{eqnarray*}
\left\{\begin{array}{rl}
     \bg_s&=\ \phi_s^*\psi^*_{\bar s}\tg_s,  \\
     \frac{\partial}{\partial s}\phi_s &=\ \nabla_{\psi^*_{\bar s}\tg_s}f_{\psi^*_{\bar s}\tg_s}\circ  \phi_s,\\
     \phi_{\bar s} &=\ \operatorname{id},
\end{array}\right.
\end{eqnarray*}
is well defined for $s\in[\bar s,\bar s+\eta]$. Since $\bar s\in[0,T)$ is arbitrary, we have that $\bg_s$ exists on $[0,T+\eta)$; this is a contradiction.
\end{proof}

Next, by Theorem \ref{thm:backward convergence}, we have that $\bg_s$ converges to a normalized Ricci shrinker $g_\infty\in\mathcal U$ with $\mu_{g_\infty}=\mu_{g_o}$ at the rate 
\begin{equation}\label{polynomialconvergence}
    \left\|\bg_s-g_\infty\right\|_{C^{k,\gamma}_{g_o}}\leq C s^{-\beta}\quad\text{ for all }\quad s\geq 1,
\end{equation}
where $C$ and $\beta$ are positive numbers.

\begin{Lemma}\label{differbydiffeo}
$g_\infty$ and $g_o$ differ only by a diffeomorphism, that is, there is a self-diffeomorphism $\psi_\infty:M^n\to M^n$, such that $\psi^*_\infty g_\infty=g_o$.
\end{Lemma}
\begin{proof}
Since $\bg_s$ and $\tg_s$ differ only by a diffeomorphism,  (\ref{SOMEC-G}) implies that
$$(M,\bg_{s_i})\longrightarrow(M^n,g_o)$$
in the smooth Cheeger-Gromov sense. On the other hand, (\ref{polynomialconvergence}) implies that
$$(M,\bg_{s_i})\longrightarrow(M^n,g_\infty)$$ also in the smooth Cheeger-Gromov sense. It follows that $g_o$ and $g_\infty$ differ only by a diffeomorphism.
\end{proof}

By (\ref{polynomialconvergence}) and Lemma \ref{differbydiffeo}, if we set $s=\log(-t)$, the we have
\begin{eqnarray}\label{anotherpolynomialconvergence}
\left\|\,\frac{1}{|t|}\psi^*_\infty\psi_{\log(-t)}^*g_t-g_o\right\|_{C^{k,\gamma}_{g_o}}\leq C(\log(-t))^{-\beta}\quad\text{ for all }\quad t\gg 1.
\end{eqnarray}
This is the conclusion of Corollary \ref{Coro_main}. 

Since Cheeger-Gromov-Hamilton convergence implies $\mathbb{F}$-convergence (\cite[Theorem 6.1]{CMZ21}), the conclusion of Theorem \ref{Thm_main} follows from the lemma below. 

\begin{Lemma}
For any sequence $\tau_i'\nearrow\infty$, we have
$$(M^n,\tau_i'^{-1}g_{\tau_i't})_{t\in(-\infty,0]}\to (M^n,g_{\infty,t})_{t\in(-\infty,0)}$$
in the smooth Cheeger-Gromov-Hamilton sense, where $g_{\infty,t}$ is the canonical form generated by $(M,g_o,f_o)$ with $g_{\infty,-1}=g_o$. Furthermore, we have
$$v(\cdot,\tau_i't)\to v_\infty(\cdot,t)$$
locally smoothly, where $v$ is the conjugate heat kernel based at any fixed point $(p_1',t_1')\in M\times(-\infty,0]$, and $v_\infty$ is the conjugate heat flow on $(M,g_{\infty,t})$ constructed using $f_o$
\end{Lemma}
\begin{proof}
(\ref{anotherpolynomialconvergence}) implies that $g_t$ has Type I curvature bound, that is,
\begin{equation*}
    \left|\Rm_{g_t}\right|\leq\frac{C}{|t|}\quad\text{ for all }\quad t>0,
\end{equation*}
where $C$ is a constant. By (\ref{Nash_bound}), $(M,g_t)$ is also noncollapsed, and hence the scaled sequence of ancient solutions has a smooth limit.

By (\ref{anotherpolynomialconvergence}), we can find a smooth diffeomorphism $\phi_\infty: M\to M$, such that $(M,\tau_i'^{-1}g_{-\tau_i'})$  $\to (M,\phi_\infty^*g_o)$ in the smooth Cheeger-Gromov sense. By the forward and backward uniqueness of the Ricci flow \cite{Ham93b, CZh09, Ko10}, we have that the Cheeger-Gromov-Hamilton limit of $(M^n,\tau_i'^{-1}g_{\tau_i't})_{t\in(-\infty,0]}$ must be $(M^n,\phi^*_\infty g_{\infty,t})$, where $g_{\infty,t}$ is the canonical form generated by $(M,g_o,f_o)$. Since $\phi_\infty$ is a time-independent diffeomorphism,  which does not affect the Cheeger-Gromov-Hamilton limit at all, we also have
$$(M^n,\tau_i'^{-1}g_{\tau_i't})_{t\in(-\infty,0)}\to (M^n,g_{\infty,t})_{t\in(-\infty,0)}$$
in the smooth Cheeger-Gromov-Hamilton sense.

For the convergence of the conjugate heat kernel, the proof is not different from that of Proposition \ref{prop: limit is shrinker}, and the details are left to the readers. Note that a compact Ricci shrinker has a unique normalized potential function.
\end{proof}

\subsection{Uniqueness of tangent flow at singular point}

The proofs of Theorem \ref{Thm_main_1} and Corollary \ref{Coro_main_1} are not essentially different from the proofs of  Theorem \ref{Thm_main} and Corollary \ref{Coro_main}. The only difference is that one needs to work with the forward modified Ricci flow instead of the backward one. We shall state the following proposition, and the reader may refer to the previous subsection for most of the details.

\begin{Proposition}\label{Pro_24}
Let $(M^n,g_t)_{t\in[-T,0)}$ be a Ricci flow satisfying the conditions in the statements of Theorem \ref{Thm_forwardrigidity}, then $t=0$ is a Type I singularity, and Theorem \ref{Thm_main_1} and Corollary \ref{Coro_main_1} hold for $(M^n,g_t)$ with $\left(\XX^\infty,(\mu^\infty_t)_{t\in -(\infty,0)}\right)$ being the canonical form of $(M,g_o,f_o)$, where $\mu^\infty_t$ is the conjugate heat flow constructed using $f_o$.
\end{Proposition}

In this case, only one lemma has a slightly different proof from the previous subsection, and we shall explain below. Indeed, by the dynamic scaling
$$\tg_s:=e^{s}g_{-e^{-s}}\quad\text{ for all }\quad s\in[0,\infty),$$
we may also argue in the same way as the previous subsection to obtain the following setting:
\begin{gather*}
    \tg_0\in\mathcal V^{k+10,\gamma}_{\delta},
    \\
    \tg_s\in\mathcal V^{k,\gamma}_{2\delta}\subset \mathcal U\quad\text{ for all }\quad s\in[0,\eta_0],
\end{gather*}
and 
\begin{eqnarray*}
\left\{\begin{array}{rl}
     \bg_s&=\ \psi_s^*\tg_s  \\
     \frac{\partial}{\partial s}\psi_s &=\ -\nabla_{\tg_s}f_{\tg_s}\circ  \psi_s,\\
     \psi_0 &=\ \operatorname{id},
\end{array}\right.
\end{eqnarray*}
is well-defined for all $s\in[0,\eta_0]$. Again, let $[0,T)$ be the maximum interval of existence of the above flow. 

\begin{Lemma}
We have $T=\infty$.
\end{Lemma}
\begin{proof}
By Theorem \ref{neighborhoods}, we have
$$\bg_s\in\mathcal V^{k,\gamma}_{\varepsilon}\quad\text{ for all }\quad s\in[0,T).$$
Here we would like to take $\varepsilon$ small enough, such that $$\mathcal V^{k,\gamma}_{2\varepsilon}\subset\mathcal U.$$

Arguing by contradiction, let us assume $T<\infty$. Since $T>\eta_0$, for all $\bar s\in[\eta_0,T)$, we can apply Theorem \ref{thm: forward pseudo} and Theorem \ref{thm: backward pseudo} to the flow $s\to \psi^*_{\bar s}\tg_s$ at $s=\bar s$. We can find $\eta_1>0$ and $C<\infty$ independent of $\bar s$, such that
$$\left|\Rm_{\psi^*_{\bar s}\tg_s}\right|_{\psi^*_{\bar s}\tg_s}\leq C\quad\text{ for all }\quad s\in[\bar s-\eta_1,\bar s+\eta_1],$$
and consequently by Shi's estimate
$$\left|\nabla^l\Rm_{\psi^*_{\bar s}\tg_s}\right|_{\psi^*_{\bar s}\tg_s}\leq C\quad\text{ for all }\quad s\in[\bar s,\bar s+\eta_1]\quad\text{ and }\quad l\leq k+10.$$
Hence, we can find a $\eta>0$ independent of $\bar s\in[\eta_0,T)$, such that
$$\psi^*_{\bar s}\tg_s\in\mathcal V^{k,\gamma}_{2\varepsilon}\subset\mathcal U \quad\text{ for all }\quad s\in[\bar s,\bar s+\eta].$$
Therefore, the flow 
\begin{eqnarray*}
\left\{\begin{array}{rl}
     \bg_s&=\ \psi^*_{s}\tg_s  \\
     \frac{\partial}{\partial s}\phi_s &=\ -\nabla_{\psi^*_{\bar s}\tg_s}f_{\psi^*_{\bar s}\tg_s}\circ  \phi_s,\\
     \phi_{\bar s} &=\ \operatorname{id},
\end{array}\right.
\end{eqnarray*}
is well define for $s\in[\bar s,\bar s+\eta]$. Since $\bar s\in[0,T)$ is arbitrary, we have that $\bg_s$ exists on $[0,T+\eta)$; this is a contradiction.
\end{proof}

The other details of the proof of Proposition \ref{Pro_24} are slight modifications of the previous subsection.

\def \tg {\widetilde{g}}

\section{Tangent flows at infinity do not depend on the basepoints}

In this section, we prove that for an $H$-concentrated metric flow, the tangent flow at infinity do not depend on the base point. For all the basic definitions involved, such as the variance, the Weissernstein distance, etc., the author may refer to \cite{Bam20b}.
We will first prove that time shifting and parabolic scaling are continuous with respect to the $\mathbb{F}$-distance, which are natural but fundamentally important. The main theorem then follows as a consequence in the same spirit as \cite[Proposition 8.2.8]{BBI01}.

Let $\XX$ be a metric flow over some $I\subset \IR.$ When necessary, we will put $\XX$ as an upper index for geometric quantities to stress that they are quantities on $\XX$. For example, $\nu^{\XX}_{x\,|\,s}$ represents the conjugate heat kernel on $\XX$ at time $s$ based at $x\in \XX$. 

We follow Cheeger's notations as in \cite{Bam20a}. We denote by
$\Psi(a_1,\cdots,a_k\,|\,b_1,\cdots, b_m)$
any function that depends on $a_1,\cdots,a_k,b_1,\cdots,b_m$ and tends to $0$ if $(a_1,\cdots,a_k)\to 0$ and $b_1,\cdots,b_m$ are fixed.
We will write $\Psi=\Psi(a_1,\cdots,a_k\,|\,b_1,\cdots, b_m)$
when there is no ambiguity and
we will add lower indices such as $\Psi_1,\Psi_2,\cdots$ to denote some other small quantities to distinguish from $\Psi.$ The exact values of these functions may vary from line to line.

As mentioned in Section 2.5, for any metric flow $\XX$ over some  $I\subset \IR$, $t_0\in \IR$, and $\lambda>0,$ we denote by
$\XX^{-t_0,\lambda}$ the metric flow obtained by first applying a $-t_0$ time shift to $\XX$ and then a parabolic rescaling by factor $\lambda$.
To be more specific, if we write
$\YY=\XX^{-t_0,\lambda}$, then
$\YY$ is a metric flow defined over
$J:=\lambda^2(I-t_0)$, such that
for each $t\in J,$ we have
\[
    \YY_{t}
    := \XX_{\lambda^{-2}t+t_0},\quad
    \dist^{\YY_t}
    := \lambda \cdot \dist^{\XX_{\lambda^{-2}t+t_0}}.
\]
For any $y\in \YY_t=\XX_{\lambda^{-2}t+t_0}$ and $s,t\in J$ with $s\le t$, we define the conjugate heat kernels by
\[
    \nu^{\YY}_{y\,|\,s}
    := \nu^{\XX}_{y\,|\,\lambda^{-2}s+t_0}.
\]
For any conjugate heat flow $(\mu_t)_{t\in I'}$ on $\XX$ over $I'\subset I$, we define
\[
    \mu^{-t_0,\lambda}_{t}
    := \mu_{\lambda^{-2}t + t_0}, \quad t\in \lambda^2(I'-t_0).
\]
For simplicity, we write
\[
    \XX^{-t_0} := \XX^{-t_0,1},\quad
    \mu^{-t_0}_t := \mu^{-t_0,1}_t
\]
for any metric flow $\XX$ and conjugate heat flow $\mu_t.$
We first prove that time shifting is continuous with respect to the $\mathbb{F}$-distance for an $H$-concentrated metric flow.

\begin{Proposition}

 \label{prop: translation continuity}
For any $H$, $V$, $T<\infty$, and $\epsilon>0,$ there is a 
$\delta=\delta(H,V,T,\epsilon)>0$ such that the following holds.
Let $\left(\XX,(\mu_t)_{t\in [-T-1,0)}\right)$ be an $H$-concentrated metric flow pair over $[-T-1,0]$.
Suppose that
\[ \sup_t {\rm Var}(\mu_t)\le V.
\]
If  $0\le \sigma\le \delta,$ 
then
\[
    \dist_{\mathbb{F}}
    \left(
    \left(
    \XX_{[-T,0]},
    (\mu_t)_{t\in [-T,0)}
    \right),
    \left(
    \XX_{[-T,0]}^{\sigma},
    (\mu_t^{\sigma})_{t\in [-T,0)}
    \right)
    \right) < \epsilon.
\]
\end{Proposition}

\textbf{Remark.} According to the definitions of metric flow pair and $\mathbb{F}$-distance, the exact form of the existence interval $I'$ of $\mu_t$ does not matter. It will be clear in the proof that we only need to assume that $|[-T,0]\setminus I'|=0,$ and the future completion (see \cite[Definition 4.42]{Bam20b}) of $I'$ is $[-T,0]$ because of the assumptions of \cite[Proposition 4.1]{Bam20b}.  For simplicity, we assumed $I'=[-T-1,0)$ in the proposition above. In applications, we will use conjugate heat kernels which exist over, e.g., $[-T-1,0].$

\begin{proof}
Let $I=[-T,0)$ and let $\beta,\delta\in (0,1/2)$ be constants to be determined later.
Then the function $v(t):={\rm Var}(\mu_t) + Ht$ is non-decreasing in $t$ by \cite[Proposition 3.34]{Bam20b}.
Let
\[
    E:=E_{\delta,\beta}:=\{ t\in [-T,0): v(t)-v(t-\delta) \ge \beta\}.
\]
Let us find a maximal finite sequence $t_1,\cdots,t_N\in E$ such that $\{[t_k-\delta,t_k]\}_{k=1}^N$ are
disjoint, then we have
\[
    A:=V+H(T+1) \ge v(0)-v(-T-1) \ge \sum_{k=1}^N v(t_k)-v(t_k-\delta)
    \ge \beta N.
\]
By the maximality of $\{t_k\}_{k=1}^N$, we have that, for any $t\in E,$
there is $1\le k\le N$ such that $[t-\delta,t]$ intersects 
$[t_k-\delta,t_k].$ Hence $t\in [t_k-\delta,t_k+\delta]$ and we have
\[
    E\subset \bigcup_{1\le k\le N} [t_k-\delta,t_k+\delta],\quad
    |E| \le 2N\delta \le \frac{2A\delta}{\beta}.
\]

Let $\tau_j=2^{-3(j+1)}/H$ for each $j\geq 0$ and we define $b:(0,1]\to(0,1)$ by
\[
    b(s):= \tfrac{1}{2}
    \Phi\left(
        - \sqrt{\tfrac{8V}{s  \tau_j}}
    \right)\quad
    \text{ for }\quad s \in (2(\tau_j H)^{1/3}, 2(\tau_{j-1} H)^{1/3}],
\]
where $\Phi:\mathbb{R}\to(0,1)$ is the same function as defined in \cite[(3.1)]{Bam20b}, satisfying
$$\Phi'(x)=(4\pi)^{-\frac{1}{2}}e^{-\frac{x^2}{4}},\quad \lim_{x\to-\infty}\Phi(x)=0,\quad \lim_{x\to\infty}\Phi(x)=1.$$
So $b$ is a positive increasing function defined on $(0,1]$.
Applying \cite[Proposition 4.1]{Bam20b} with $r=1$, we have that, given $\sigma\in[0,\delta]$, for each $t\in I\setminus \big(E\cup(-H^{-1}\sigma^3,0]\big),$ it holds that 
\begin{eqnarray}\label{massconcentrationlowerbound}
    b^{(\XX_t,\dist_t,\mu_t)}_1(\varepsilon)
    \ge b(\varepsilon) \quad
    \text{ for all } \quad\varepsilon\in[\sigma,1],
\end{eqnarray}
where $b_r^{(X,d,\mu)}:(0,1]\to(0,1]$ is the mass distribution function at scale $r>0$; see \cite[Definition 2.17]{Bam20b}. 

Next, we shall apply \cite[Proposition 4.14]{Bam20b}. To this end, we verify that the assumptions therein are satisfied by $\XX$ with $t\in I\setminus E$ and $t'=t-\sigma$, where $\sigma\in[0,\delta]$. Indeed, applying \cite[Lemma 4.7]{Bam20b} and the definition of $E$, we have
\begin{align}\label{thevariancesmalldifference}
 \int_{\XX_t\times \XX_t} \dist_t\, d\mu_t d\mu_t
&- \int_{\XX_{t'}\times \XX_{t'}} \dist_{t'}\, d\mu_{t'} d\mu_{t'}
\le \sqrt{v(t)-v(t')} + 2\sqrt{H(t-t')}\\\nonumber
&\le \sqrt{v(t)-v(t-\delta)} + 2\sqrt{H(t-t')}\\\nonumber
&\le \sqrt{\beta} + 2\sqrt{H\delta}
< \Psi(\delta,\beta\,|\,H).
\end{align}
Given (\ref{massconcentrationlowerbound}) and (\ref{thevariancesmalldifference}), we can now apply \cite[Proposition 4.14]{Bam20b} with $r=1$, and conclude that for each $t\in I\setminus \big(E\cup(-H^{-1}\delta^3,0]\big)$ and $\sigma\in[0,\delta]$, writing $t'=t-\sigma r^2=t-\sigma\in[t-\delta,t]$, 
there is a closed subset $W_t\subset \XX_t$ such that:
\begin{enumerate}[(1)]
    \item $\mu_t(\XX_t\setminus W_t)<\Psi(\delta,\beta\,|\,H,V)$.
    
    \item For any $y_1,y_2\in W_t,$ we have
    \[
        0\leq \dist_t(y_1,y_2) - \dist^{\XX_{t'}}_{W_1}(\nu_{y_1\,|\,t'},\nu_{y_2\,|\,t'})
        <\Psi(\delta,\beta\,|\,H,V).
    \]
    \end{enumerate}
Furthermore, there exist a metric space $(Z_t,\dist^{Z_t})$ and isometric embeddings
$\psi_t:\XX_t\to Z_t$,  $\phi_t:\XX_{t'}\to Z_t,$ such that:
\begin{enumerate}[(1)]
\setcounter{enumi}{2}
\item For any $x\in \XX_{t'},y\in W_t,$ we have
\[
    \dist^{Z_t}(\phi_t(x),\psi_t(y))
    \le \dist^{\XX_{t'}}_{W_1}(\delta_x, \nu_{y\,|\,t'}) + \Psi(\delta,\beta\,|\,H,V).
\]

\item We can construct the following coupling between $\mu_t$ and $\mu_{t'}$
\[
q_t
:= \int_{\XX_t} (\nu_{y\,|\,t'}\otimes \delta_y)\, d\mu_t(y)
\in \Pi(\mu_{t'}, \mu_{t})
\]
satisfying the estimate
\[
    \dist^{Z_t}_{W_1}
    (\phi_{t*}\mu_{t'},\psi_{t*}\mu_t)\leq\int_{\XX_{t'}\times\XX_{t}}\dist^{Z_t}(\phi_t(x),\psi_t(y))dq_t(x,y)
    < \Psi(\delta,\beta\,|\,H,V).
\]
\end{enumerate}
For $t\in I\cap\big(E\cup(-H^{-1}\delta^3,0]\big),$ we let $(Z_t,\dist^{Z_t})$ be an arbitrary separable metric space into which
$(\XX_t,\dist_t)$ and $(\XX_{t-\sigma},\dist_{t-\sigma})$ can be embedded.
Now $\mathfrak{C}=\left((Z_t,\dist^{Z_t}),(\psi_t,\phi_t)\right)_{_{t\in I}}$ serves as a correspondence between $\XX_{I}$ and $\XX_{I}^{\sigma}.$

The goal is to show that for each $s\le t,s,t\in I\setminus \big(E\cup(-H^{-1}\delta^3,0]\big),$ it holds that
\begin{equation*}
   \text{if } \delta\le \bar \delta(\epsilon,H,V,T), \quad \text{ then } \int_{\XX^{\sigma}_{t}\times \XX_{t}} \dist^{Z_s}_{W_1}
    \left(
    \phi_{s*}\nu_{x^1\,|\,s-\sigma},
    \psi_{s*}\nu_{x^2\,|\,s}
    \right)
    d q_t(x^1,x^2)
    < \epsilon,
\end{equation*}
where $q_t$ is the coupling defined in item (4) above. We need only to verify the case where $s<t,$ because the equality case is equivalent to item (4) above. We write
\[
    s':=s-\sigma,\quad t':=t-\sigma.
\]
Note that
\begin{align*}
   & \int_{\XX^{\sigma}_{t}\times \XX_{t}} \dist^{Z_s}_{W_1}
    \left(
    \phi_{s*}\nu_{x\,|\,s'},
    \psi_{s*}\nu_{y\,|\,s}
    \right)
    d q_t(x,y)\\
=& \int_{\XX_{t}} \int_{\XX_{t}^{\sigma}} 
\dist^{Z_s}_{W_1}
    \left(
    \phi_{s*}\nu_{x\,|\,s'},
    \psi_{s*}\nu_{y\,|\,s}
    \right)
    \, d\nu_{y\,|\,t'}(x)
    \,d\mu_t(y)\\
\le& 
\int_{\XX_{t}} \int_{\XX_{t}^{\sigma}} 
\dist^{\XX_{s'}}_{W_1}
(\nu_{x\,|\,s'}, \nu_{y\,|\,s'})
    \, d\nu_{y\,|\,t'}(x)
    \,d\mu_t(y)\\
   &\quad + \int_{\XX_{t}} \int_{\XX_{t}^{\sigma}} 
\dist^{Z_s}_{W_1}
    \left(
    \phi_{s*}\nu_{y\,|\,s'},
    \psi_{s*}\nu_{y\,|\,s}
    \right)\, d\nu_{y\,|\,t'}(x)
    \,d\mu_t(y)\\
=:&\ I_1 + I_2.
\end{align*}

For $I_1,$ by the monotonicity formula \cite[Propsition 3.24(b)]{Bam20b}, the definition of the $W_1$-Weissernstein distance, and the definition of variance, we have
\begin{align*}
    I_1
    &\le \int_{\XX_{t}} \int_{\XX_{t}^{\sigma}} 
    \dist^{\XX_{t'}}_{W_1}(\delta_x, \nu_{y\,|\,t'})
     \, d\nu_{y\,|\,t'}(x)
    \,d\mu_t(y) \\
    & = \int_{\XX_t}d\mu_t(y)
    \int_{\XX_{t'}\times\XX_{t'}} \dist_{t'}\, d\nu_{y\,|\,t'}d\nu_{y\,|\,t'}\\
    &\le \int_{\XX_{t}} 
    {\rm Var}(\nu_{y\,|\,t'})^{1/2}
    \,d\mu_t(y)
    \le \sqrt{H \delta},
\end{align*}
where we  used \cite[Proposition 3.34]{Bam20b} in the last inequality.

$I_2$ can be simplified as
\begin{align*}
    I_2
    &= \int_{\XX_{t}}
    \dist^{Z_s}_{W_1}
    \left(
    \phi_{s*}\nu_{y\,|\,s'},
    \psi_{s*}\nu_{y\,|\,s}
    \right)
    \,d\mu_t(y).
\end{align*}
We shall use the same argument as in the proof of \cite[Lemma 4.18]{Bam20b}
to obtain an estimate of $I_2$. Fix any $x\in \XX_{s'}$ and $w\in W_{s}^\delta\subset \XX_s,$ where we let
\[
    W^{\delta}=B(W,\delta),
\]
following the notations in \cite[Lemma 4.18]{Bam20b}. 
Let $w'\in W_s$ such that $\dist_s(w,w')<\delta$, then, by item (3) above, we have
\begin{align}\label{summanonsense}
\nonumber
    \dist^{Z_s}(\phi_s(x),\psi_s(w))
    &\le \dist^{Z_s}(\phi_s(x),\psi_s(w')) + \delta
    \le \dist^{\XX_{s'}}_{W_1}(\delta_x, \nu_{w'\,|\,s'}) + \Psi(\delta,\beta\,|\,H,V)\\
    &\le  \dist^{\XX_{s'}}_{W_1}(\delta_x, \nu_{w\,|\,s'})
    + \dist_s(w,w')+ \Psi(\delta,\beta\,|\,H,V)\\\nonumber
    &\le \dist^{\XX_{s'}}_{W_1}(\delta_x, \nu_{w\,|\,s'}) + \Psi(\delta,\beta\,|\,H,V),
\end{align}
where we have applied \cite[Propsition 3.24(b)]{Bam20b} again. Define 
\[
    q:=q_{y,s}
    := \int_{\XX_s} \nu_{z\,|\,s'}\otimes \delta_z\, 
    d\nu_{y\,|\,s}(z)
    \in \Pi(\nu_{y\,|\,s'},\nu_{y\,|\,s})\quad \text{ for each }\quad  y\in \XX_t,
\]
we may use the definition of the $W_1$-Weissernstein distance to estimate the integrand of $I_2$ as follows.
\begin{align*}
    \dist^{Z_s}_{W_1}
    \left(
    \phi_{s*}\nu_{y\,|\,s'},
    \psi_{s*}\nu_{y\,|\,s}
    \right)
    &\le  
    \int_{\XX_{s'}\times \XX_{s}}
    \dist^{Z_s}(\phi_{s}(x_1), \psi_s(x_2))
    \, dq(x_1,x_2)\\
   & =
   \int_{\XX_{s}}
    \int_{\XX_{s'}}
    \dist^{Z_s}(\phi_{s}(x), \psi_s(z))
    \, d\nu_{z\,|\,s'}(x)
    \, d\nu_{y\,|\,s}(z)\\
    & =
    \int_{W_s^{\delta}}\int_{\XX_{s'}}
    + \int_{\XX_s\setminus W_s^{\delta}}\int_{\XX_{s'}}
    =: A_1(y) + A_2(y).
\end{align*}
On one hand, by \eqref{summanonsense} and applying \cite[Proposition 3.34]{Bam20b} again, we have
\begin{align*}
  A_1(y) = &  \int_{W_s^{\delta}}
    \int_{\XX_{s'}}
    \dist^{Z_s}(\phi_{s}(x), \psi_s(w))
    \, d\nu_{w\,|\,s'}(x)
    \, d\nu_{y\,|\,s}(w)\\
\le & \int_{W_s^{\delta}}
    \int_{\XX_{s'}}
        \dist^{\XX_{s'}}_{W_1}
        (\delta_x, \nu_{w\,|\,s'})
    \, d\nu_{w\,|\,s'}(x)
    \, d\nu_{y\,|\,s}(w) + \Psi(\delta,\beta\,|\,H,V)\\
=&\int_{W_s^{\delta}}
    \left(\int_{\XX_{s'}\times \XX_{s'}}
    \dist_s(x,x')
    \,d\nu_{w\,|\,s'}(x)
    \,d\nu_{w\,|\,s'}(x')
    \right) d\nu_{y\,|\,s}(w) + \Psi(\delta,\beta\,|\,H,V)\\
\le& 
\int_{W_s^{\delta}}
    {\rm Var}(\nu_{w\,|\,s'})^{1/2}
    \, d\nu_{y\,|\,s}(w) + \Psi(\delta,\beta\,|\,H,V)\\
\le& \sqrt{H\delta} + \Psi(\delta,\beta\,|\,H,V).
\end{align*}
On the other hand, 
if $\delta,\beta$ are sufficiently small, then
$\mu_s(W_s^{\delta})\ge 1/2$, and we may compute as follows using (\ref{summanonsense}) and \cite[Proposition 3.34]{Bam20b}.
\begin{align}\label{superextranonsense}
    A_2(y)/2=
    &\ \frac{1}{2}
    \int_{\XX_s\setminus W_s^{\delta}} 
    \int_{\XX_{s'}}
    \dist^{Z_s}(\phi_{s}(x), \psi_s(z))
    \, d\nu_{z\,|\,s'}(x)
    \, d\nu_{y\,|\,s}(z)\\\nonumber
\le &  \,
    \int_{W_s^{\delta}}
    \int_{\XX_s\setminus W_s^{\delta}} 
    \int_{\XX_{s'}}
    \dist^{Z_s}(\phi_{s}(x), \psi_s(z))
    \, d\nu_{z\,|\,s'}(x)
    \, d\nu_{y\,|\,s}(z)
    \, d\mu_s(w)\\\nonumber
\le& \,
\int_{W_s^{\delta}}
    \int_{\XX_s\setminus W_s^{\delta}} 
    \int_{\XX_{s'}}
    \big\{\dist^{Z_s}(\phi_{s}(x), \psi_s(w))
    + \dist_s(w,z) \big\}
    \, d\nu_{z\,|\,s'}(x)
    \, d\nu_{y\,|\,s}(z)
    \, d\mu_s(w)\\\nonumber
\le& \,
\int_{W_s^{\delta}}
    \int_{\XX_s\setminus W_s^{\delta}} 
    \int_{\XX_{s'}}
    \dist^{\XX_{s'}}_{W_1}
    (\delta_x, \nu_{w\,|\,s'})
    \, d\nu_{z\,|\,s'}(x)
    \, d\nu_{y\,|\,s}(z)
    \, d\mu_s(w) + \Psi(\delta,\beta\,|\,H,V)\\\nonumber
& \, 
+ \int_{W_s^{\delta}}
    \int_{\XX_s\setminus W_s^{\delta}}
     \dist_s(w,z) 
    \, d\nu_{y\,|\,s}(z)
    \, d\mu_s(w)\\\nonumber
=& \,
\int_{W_s^{\delta}}
    \int_{\XX_s\setminus W_s^{\delta}} 
    \left(\int_{\XX_{s'}\times\XX_{s'}}
    \dist^{\XX_{s'}}
    (x,x')
    \, d\nu_{z\,|\,s'}(x)\,d\nu_{w\,|\,s'}(x')\right)
    \, d\nu_{y\,|\,s}(z)
    \, d\mu_s(w) \\\nonumber
& \, 
+ \int_{W_s^{\delta}}
    \int_{\XX_s\setminus W_s^{\delta}}
     \dist_s(w,z) 
    \, d\nu_{y\,|\,s}(z)
    \, d\mu_s(w)+ \Psi(\delta,\beta\,|\,H,V) \\\nonumber
\le& \,
\int_{W_s^{\delta}}
    \int_{\XX_s\setminus W_s^{\delta}} 
    {\rm Var}(\nu_{z\,|\,s'},\nu_{w\,|\,s'})^{1/2}
    \, d\nu_{y\,|\,s}(z)
    \, d\mu_s(w)
\\\nonumber
&\, + \int_{W_s^{\delta}}
    \int_{\XX_s\setminus W_s^{\delta}}
     \dist_s(w,z) 
    \, d\nu_{y\,|\,s}(z)
    \, d\mu_s(w) +\Psi(\delta,\beta\,|\,H,V)\\\nonumber
\le& \, 
\int_{W_s^{\delta}}
    \int_{\XX_s\setminus W_s^{\delta}}
     \left(\sqrt{\dist_s^2(w,z) + H\delta}+\dist_s(w,z) \right)
     d\nu_{y\,|\,s}(z)
    \, d\mu_s(w) +\Psi(\delta,\beta\,|\,H,V) \\\nonumber
\le& \, \left(2\int_{W_s^{\delta}}
    \int_{\XX_s\setminus W_s^{\delta}}\left(2\dist_s^2(w,z) + H\delta\right)\, d\nu_{y\,|\,s}(z)
    \, d\mu_s(w) \right)^{\frac{1}{2}}\left(\int_{W_s^{\delta}}
    \int_{\XX_s\setminus W_s^{\delta}} d\nu_{y\,|\,s}(z)
    \, d\mu_s(w) \right)^{\frac{1}{2}} \\\nonumber
    &\, + \Psi(\delta,\beta\,|\,H,V)\\\nonumber
\le& \, 
2 \left\{
    \mu_s(W_s^{\delta})\cdot
    \nu_{y\,|\,s}(\XX_s\setminus W_s^{\delta})
\right\}^{1/2}
\left({\rm Var}(\nu_{y\,|\,s},\mu_s)+\tfrac{1}{2}H\delta\right)^{1/2} + \Psi(\delta,\beta\,|\,H,V),
\end{align}
where we have also applied the definition of $H$-concentration. Since
\[
    \int_{\XX_t} \nu_{z\,|\,s}(\XX_s\setminus W_s^{\delta})\, d\mu_t(z)
    = \mu_s(\XX_s\setminus W_s^{\delta}) < \Psi(\delta,\beta\,|\,H,V),
\]
taking $\Psi_1=\sqrt{\Psi}$ and $\Omega_t:=\{\nu_{\cdot\,|\,s}(\XX_s\setminus W_s^{\delta}) < \Psi_1\}$, we have
\[
    \mu_t(\Omega_t) \ge 1-\Psi_1(\delta,\beta\,|\,H,V),\quad
    \nu_{z\,|\,s}(\XX_s\setminus W_s^{\delta}) < \Psi_1(\delta,\beta\,|\,H,V)\ \text{ for each }\ z\in \Omega_t.
\]
It follows from (\ref{superextranonsense}) that
\begin{align*}
    &\int_{\XX_t} A_2(y) \, d\mu_t(y)
= \int_{\Omega_t} + \int_{\XX_t\setminus \Omega_t}\\
    \le&\  4\Psi_1(\delta,\beta\,|\,H,V)^{1/2} 
    \int_{\XX_t} \left({\rm Var}(\nu_{y\,|\,s},\mu_s) +\tfrac{1}{2}H\delta\right)^{1/2}d\mu_t(y)
    \\
    &\, + \Psi (\delta,\beta\,|\,H,V)
 + 4\int_{\XX_t\setminus \Omega_t}\left( {\rm Var}(\nu_{y\,|\,s},\mu_s)+\tfrac{1}{2}H\delta\right)^{\frac{1}{2}} d\mu_t(y) \\
\le& \ 
4\Psi_1(\delta,\beta\,|\,H,V)^{1/2}\left(
    \int_{\XX_t} {\rm Var}(\delta_y,\mu_t)\, d\mu_t(y)
    + H(t-s)+\tfrac{1}{2}H\delta
    \right)^{1/2} + \Psi(\delta,\beta\,|\,H,V) \\
&\ + 4\mu_t(\XX_t\setminus \Omega_t)^{1/2}
\left(
    \int_{\XX_t} {\rm Var}(\delta_y,\mu_t)\, d\mu_t(y)
    + H(t-s)+\tfrac{1}{2}H\delta
    \right)^{1/2} \\
\le & \ 
10\Psi_1(\delta,\beta\,|\,H,V)^{1/2}(V+HT+H\delta)^{1/2} + \Psi(\delta,\beta\,|\,H,V) < \Psi_2(\delta,\beta\,|\,H,V,T).
\end{align*}

Combining the estimates above, we have
\[
\int_{\XX^{\sigma}_{t}\times \XX_{t}} \dist^{Z_s}_{W_1}
    \left(
    \phi_{s*}\nu_{x^1\,|\,s'},
    \psi_{s*}\nu_{x^2\,|\,s}
    \right)
    d q_t(x^1,x^2) < \Psi_2(\delta,\beta\,|\,H,V,T).
\]
For a given $\epsilon>0$, we may first choose $\beta\le \bar \beta(\epsilon, H,V,T)$ and $\delta\le \bar \delta(\epsilon,H,V,T)$ so that $\Psi_2$ above satisfies
$\Psi_2< \epsilon.$
Then choose $\delta \le \bar \delta'(\beta,\epsilon, H,V,T)$,
such that
\[
    |E\cup(-H^{-1}\delta^3,0]| \le \frac{2A\delta}{\beta}+H^{-1}\delta^3 < \epsilon^2.
\]
(Recall that $A=V+H(T+1).$)
Then, by the definition of the $\mathbb{F}$-distance, we have
\[
\dist_{\mathbb{F}}
    \left(
    \left(
    \XX_{[-T,0)},
    (\mu_t)_{t\in [-T,0)}
    \right),
    \left(
    \XX_{[-T,0)}^{\sigma},
    (\mu_t^{\sigma})_{t\in [-T,0)}
    \right)
    \right) < \epsilon.
\]

\end{proof}

Using essentially the same arguments as above, we can show that the operation of parabolic scaling is continuous at scale $1$ (and hence at any scale). We do not need the following proposition in this article, and the detailed proof is left to the reader.

\begin{Proposition}
\label{prop: scaling continuity}
For any $H$, $V$, $T<\infty$ and $\epsilon>0,$ there is 
$\delta=\delta(H,V,T,\epsilon)>0$ suth that the following holds.
Let $(\XX,(\mu_t))$ be an $H$-concentrated metric flow pair over $[-T-1,0]$.
Suppose 
\[ \sup_t {\rm Var}(\mu_t)\le V.
\]
If  $|\lambda-1|\le \delta,$ 
then
\[
    \dist_{\mathbb{F}}
    \left(
    \left(
    \XX_{[-T,0]},
    (\mu_t)_{t\in [-T,0)}
    \right),
    \left(
    \XX_{[-T,0]}^{0,\lambda},
    (\mu_t^{0,\lambda})_{t\in [-T,0)}
    \right)
    \right) < \epsilon.
\]
\end{Proposition}

\begin{Proposition} \label{prop: local continuity}
For any $H$, $T<\infty$ and $\epsilon>0,$ there is 
$\delta=\delta(\epsilon,H,T)>0$ such that the following holds.
Let $\XX$ be an $H$-concentrated metric flow over $(-\infty,1)$
and $x_0\in \XX_{0}.$ 
If $\sigma\in (0,\delta)$ and
$y_0\in \PP^{*}(x_0;\delta)\cap \XX_{-\sigma},$
then
\[
    \dist_{\mathbb{F}}
    \left(
    \left(
    \XX_{[-T,0]},
    (\nu_{x_0\,|\,t})_{t\in [-T,0]}
    \right),
    \left(
    \XX_{[-T,0]}^{\sigma},
    (\nu_{y_0\,|\,t}^{\sigma})_{t\in [-T,0]}
    \right)
    \right) < \epsilon.
\]
Here $\PP^*$ (as well as $\PP^{*+}$ and $\PP^{*-}$ in the proof below) is the $W_1$-parabolic neighborhood defined in \cite[Definition 3.38, Definition 3.39]{Bam20b}.
\end{Proposition}
\begin{proof}
We may assume that $y_0\in \PP^{*-}(x_0;\delta)$ because if 
$y_0\in \PP^{*+}(x_0;\delta),$ then $x_0\in \PP^{*-}(y_0;\delta)$
and we switch the role of $x$ and $y$.

Let 
$\delta=\delta_{\ref{prop: translation continuity}}(H,H(T+1),T,\epsilon/2)>0$
be given by Proposition \ref{prop: translation continuity} such that
\[
    \dist_{\mathbb{F}}
    \left(
    \left(
    \XX_{[-T,0]},
    (\nu_{x_0\,|\,t})_{t\in [-T,0]}
    \right),
    \left(
    \XX_{[-T,0]}^{\sigma},
    (\nu_{x_0\,|\,t}^{\sigma})_{t\in [-T,0]}
    \right)
    \right) < \epsilon/2.
\]
By the monotonicity of the $W_1$-Weinssernstein distance  \cite[Proposition 3.24(b)]{Bam20b} and the definition of $\PP^{*-}(x_0,\delta)$,
for each $t\in [-T,-\delta^2],$ we have
\[
\dist^{\XX_t^{\sigma}}_{W_1}
(\nu_{x_0\,|\,t}^{\sigma},\nu_{y_0\,|\,t}^{\sigma})
\le \dist^{\XX_{-\delta^2}}_{W_1}
(\nu_{x_0\,|\,-\delta^2},\nu_{y_0\,|\,-\delta^2})
< \delta.
\]
By \cite[Lemma 5.19]{Bam20b}, if $\delta<\epsilon/2$, then we have
\[
    \dist_{\mathbb{F}}
    \left(
    \left(
    \XX_{[-T,0]}^{\sigma},
    (\nu_{x_0\,|\,t}^{\sigma})_{t\in [-T,0]}
    \right),
    \left(
    \XX_{[-T,0]}^{\sigma},
    (\nu_{y_0\,|\,t}^{\sigma})_{t\in [-T,0]}
    \right)
    \right) < \epsilon/2,
\]
The conclusion follows from the triangle inequality of $\dist_{\mathbb{F}}.$

\end{proof}

We are now ready to prove Theorem \ref{thm: rough tangent} and Theorem \ref{thm: tangent flows}.

\begin{proof}[Proof of Theorem \ref{thm: rough tangent} and Theorem \ref{thm: tangent flows}]
Suppose 
$x_0\in \XX_{t_0}$ and $y_0\in \XX_{s_0}$ with $s_0\le t_0.$
Suppose that $\lambda_j\to 0$
is a sequence such that 
\[
\left(
    \XX_{[-T,0]}^{-t_0,\lambda_j},
    (\nu_{x_0\,|\,t}^{-t_0,\lambda_j})_{t\in [-T,0]}
    \right)
\xrightarrow[j\to \infty]{\mathbb{F}} 
\left(
    \XX_{[-T,0]}^{\infty},
    (\nu_{x_{\max}\,|\,t}^{\infty})_{t\in [-T,0]}
    \right)
    \in \TT_{x_0},
\]
for each $T<\infty$.
\def \YY {\mathcal{Y}}

Suppose that $y_0\in \PP^{*}(x_0;\rho),$
for some $\rho<\infty.$ In fact, such a number $\rho$ exists because we may take $z_0$ to be an $H$-center of $x_0$ at time $s_0$ and take $\rho$ to be any large number so that  \[
\rho> \dist_{s_0}(y_0,z_0) + \sqrt{H(t_0-s_0)} + \sqrt{t_0-s_0}.
\]
Then
\begin{align*}
    &\dist_{W_1}^{\XX_{t_0-\rho^2}}(\nu_{y_0\,|\,t_0-\rho^2},\nu_{x_0\,|\,t_0-\rho^2})
    \le \dist_{W_1}^{\XX_{s_0}}(\delta_{y_0},\nu_{x_0\,|\,s_0})\\
    \le&\  \dist_{s_0}(y_0,z_0)
    + \dist_{W_1}^{\XX_{s_0}}(\delta_{z_0},\nu_{x_0\,|\,s_0})
    \le \dist_{s_0}(y_0,z_0)
    + \sqrt{H(t_0-s_0)} < \rho.
\end{align*}
If we write
\[
    \tilde\XX^j := \XX^{-t_0,\lambda_j},\quad
    \tilde\nu^j_{x\,|\,t} := \nu^{-t_0,\lambda_j}_{x\,|\,t}
    = \nu_{x\,|\, t_0 + \lambda_j^{-2}t},
\]
then
\[
    x_0\in \tilde\XX^j_0,
    \quad
    \XX^{-s_0,\lambda_j}
    =(\tilde\XX^j)^{\lambda_j^2(t_0-s_0)},\quad
    y_0\in \PP^{*}_{\tilde\XX^j}(x_0;\lambda_j\rho)
    \cap \tilde\XX^j_{-\lambda_j^2(t_0-s_0)}.
\]
Note that $H$-concentration is invariant under parabolic rescaling. Hence for any $ \epsilon>0,$ by Proposition \ref{prop: local continuity}, there is $\delta>0$ depending on $\epsilon, T, H$ such that
if $j$ is sufficiently large so that $\lambda_j\rho<\delta,$ then
\begin{align*}
    &
    \dist_{\mathbb{F}}
    \left(
    \left(
    \XX_{[-T,0]}^{-t_0,\lambda_j},
    (\nu_{x_0\,|\,t}^{-t_0,\lambda_j})_{t\in [-T,0]}
    \right),
    \left(
    \XX_{[-T,0]}^{-s_0,\lambda_j},
    (\nu_{y_0\,|\,t}^{-s_0,\lambda_j})_{t\in [-T,0]}
    \right)
    \right)\\
= & \, \dist_{\mathbb{F}}
    \left(
    \left(
    \tilde\XX^j_{[-T,0]},
    (\tilde\nu^j_{x_0\,|\,t})_{t\in [-T,0]}
    \right),
    \left(
    \left(\tilde\XX^j_{[-T,0]}\right)^{\lambda_j^2(t_0-s_0)},
    \left(\left(\tilde \nu^j_{y_0\,|\,t}\right)^{\lambda_j^2(t_0-s_0)}\right)_{t\in [-T,0]}
    \right)
    \right) < \epsilon.
\end{align*}
It follows from taking $\epsilon\to 0$ that $\TT_{x_0}^\infty\subset \TT_{y_0}^\infty.$
The proof of the other direction is the same.
\end{proof}

\appendix
\appendixpage

\section{Proof of Theorem 2.3}

The proof is not essentially different from \cite[Lemma 3.2]{SW15}. Let $\varepsilon>0$ be an arbitrary small number and let $\delta\in(0,\frac{1}{2}\varepsilon)$ which will be specified later. We first of all assume that $\mathcal V_{2\varepsilon}^{k,\gamma}\subset \mathcal U$, where $\mathcal U$ is a regular neighborhood of $g_o$. Now, let $(M,g_t)_{t\in[0,T)}$ be a backward modified Ricci flow with $g_0\in \mathcal V_{\delta}^{k,\gamma}$ and $[0,T)$ the maximum interval of existence.

\begin{Lemma}\label{shorttimestability}
 There exists a small $\eta>0$,
 such that $g_t\in \mathcal V_{2\delta}^{k,\gamma}$ for all $t\in[0,\eta]$.
\end{Lemma}

\begin{proof}
Since the backward modified Ricci flow and the backward Ricci flow differ only by a one-parameter family of diffeomorphisms and a time scaling, we can consider 
\begin{eqnarray}\label{supralapsarianism}
\left\{\begin{array}{rl}
     \tilde g_t &=\ \psi_t^* g_t,  \\
     \partial_t \psi_t &=\ -\nabla_{g_t}f_{g_t}\circ\psi_t, \\
     \psi_0&=\ \operatorname{id}.
\end{array}\right.
\end{eqnarray}
Note that the existence of $g_t$ presupposes that $\nabla f_{g_t}$ is well-defined for all $t\in[0,T)$. Then $\tilde g$ is a normalized backward Ricci flow satisfying $\partial_t \tilde g_t=2(\Ric_{\tilde g_t}-\frac{1}{2}\tilde g_t)$ and $\tilde g_0=g_0$.
Applying Bamler's backward pseudolocality at $t=0$ to $\tilde g$, we have that, there is an $\eta_0>0$, such that
$$\left|\Rm_{\tilde g_t}\right|_{\tilde g_t}\leq C(\eta_0)\quad \text{ for all }\quad t\in[0,\eta_0].$$
By applying Shi's estimates, we also have that all derivatives of the curvature are bounded on $M\times[0,\frac{\eta_0}{2}]$ in terms of $C(\eta_0)$. Therefore, taking $\eta_1\leq \overline \eta_1(\eta_0,\delta)$, we have
$$\tilde g_t\in\mathcal V_{2\delta}^{k,\gamma}\quad\text{ for all }\quad t\in[0,\eta_1],$$
and consequently 
$$\|f_{g_t}\|_{C^{l,\gamma}_{g_t}}=\|f_{\tilde g_t}\|_{C^{l,\gamma}_{\tilde g_t}}\leq C(l) \quad\text{ for all }\quad t\in[0,\eta_1]\quad\text{ and for all }\quad l\in\mathbb{N}.$$
This further implies that
$$\left\|\Ric_{g_t}+\nabla^2_{g_t}f_{g_t}-\frac{1}{2}g_t\right\|_{C^{l,\gamma}_{g_t}}=\left\|\Ric_{\tilde g_t}+\nabla^2_{\tilde g_t}f_{\tilde g_t}-\frac{1}{2}\tilde g_t\right\|_{C^{l,\gamma}_{\tilde g_t}}\leq C(l),$$
for all $t\in[0,\eta_1]$ and for all $l\in\mathbb{N}$. Finally, if we take $\eta<\overline\eta(\delta,\eta_1)$ small enough, then the conclusion of the lemma follows.
\end{proof}

\begin{Lemma}\label{longtimecontinuity}
By taking $\delta<\overline\delta(\varepsilon)$ small enough, we have that, for all $T_0\in(0,T-\varepsilon]$, if $g_t\in \mathcal V_{\varepsilon}^{k,\gamma}$ for all $t\in[0,T_0)$, then $g_{T_0}\in\mathcal V_{\frac{\varepsilon}{2}}^{k,\gamma}$.
\end{Lemma}

\begin{proof}
Since $T_0< T$, we may apply Bamler's backward pseudolocality theorem as in the proof of the above lemma, that is, for all $t_0\in[0,T_0)$, we may define a flow $\tilde g$ as in (\ref{supralapsarianism}) with $\psi_{t_0}=\operatorname{id}$ instead, and apply Theorem \ref{thm: backward pseudo} to $\tilde g_t$ at $t_0$. Then we can find a positive number $\eta>0$, independent of $T_0$ and $t_0$ (note that since $g_{t_0}$ is very close to $g_o$, hence we can estimate all of its geometric quantities using $g_o$, and this is why $\eta$ is independent of $T_0$ and $t_0$), such that, for each $t_0\in[0,T_0)$, we have
$$\left|\nabla^l\Rm_{g_t}\right|_{g_t}\leq C(l),\quad \left\|\Ric_{g_t}+\nabla_{g_t}^2f_{g_t}-\frac{1}{2}g_t\right\|_{C^{l,\gamma}_{g_t}}\leq C(l),$$
for all $t\in[t_0,t_0+\eta]$ and $l\in\mathbb{N}$, where the constants $C(l)$ are also independent of $T_0$ and $t_0$. Note that here we have also applied $\mathcal V_{2\varepsilon}^{k,\gamma}\subset\mathcal U$ as in the above lemma. It is easily observed that
$$g_t\in\mathcal V_{2\varepsilon}^{k,\gamma}\quad\text{ for all }\quad t\in[0,T_0].$$

As in  \cite[Lemma 3.2]{SW15}, for all integer $p\ge 1$ and for all $t\in[0,T_0)$, we obtain, by the standard interpolation inequality, that,
\begin{eqnarray*}
    \left\|
    \partial_t g_t
    \right\|_{W^{p,2}_{g_t}}
    &\le& C(p) \left\|
    \partial_t g_t
    \right\|_{L^{2}_{g_t}}^{\beta}\left\|\Ric_{g_t}+\nabla_{g_t}^2f_{g_t}-\frac{1}{2}g_t\right\|^{1-\beta}_{W^{N(p), 2}_{g_t}}
    \\
    &\leq &C(p)\left\|
    \partial_t g_t
    \right\|_{L^{2}_{g_t}}^{\beta},
\end{eqnarray*}
where $\beta\in( 2- 1/\alpha, 1)$ and $p\ll N(p)\in\mathbb{N}$.  Here $\alpha\in[\frac{1}{2},1)$  is the constant in Theorem \ref{Lojaciewicz}. 
Since the Sobolev constants are uniformly bounded for metrics in $\mathcal{V}^{k,\gamma}_{2\varepsilon}$, we have
\[
    \left\|
    \partial_t g_t
    \right\|_{C^{l,\gamma}_{g_t}}
    \le C(l)\left\|
    \partial_t g_t
    \right\|_{W^{p, 2}_{g_t}}
    \le C(l) \left\|
    \partial_t g_t
    \right\|_{L^{2}_{g_t}}^{\beta}\quad\text{ for all }\quad t\in[0,T_0].
\]
By the definition of backward modified Ricci flow and Theorem \ref{Lojaciewicz}, we have
\begin{align*}
    \frac{d}{dt}  \left( \mu_{g_t} - \mu_{g_o}\right)^{1-(2-\beta)\alpha}
    &= -2 (1-(2-\beta)\alpha) \left( \mu_{g_t} - \mu_{g_o}\right)^{-(2-\beta)\alpha} \|\nabla \mu_{g_t}\|^2_{L^2_{g_t}}\\
    & \le - C(\beta) \left\|
    \partial_t g_t
    \right\|_{L^{2}_{g_t}}^{\beta}.
\end{align*}
Then
\[
\left\|
    \partial_t g_t
    \right\|_{C^{l,\gamma}_{g_t}}
    \le -C(\beta,l)  \frac{d}{dt}  \left( \mu_{g_t} - \mu_{g_o}\right)^{1-(2-\beta)\alpha}.
\]
Hence, for all $t\in (0,T_0]$, we have
\begin{align}\label{infralapsarianism}
    \| g_t - g_0\|_{C^{k,\gamma}_{g_o}}
    &\le C\int_{0}^{t} \left\|
    \partial_s g_s
    \right\|_{C^{k,\gamma}_{g_s}}\, ds\\\nonumber
    &\le C(\beta) \left( \mu_{g_0}- \mu_{g_o}\right)^{1-(2-\beta)\alpha}\\\nonumber
    & < \varepsilon/4,
\end{align}
if we take $\delta \le \bar \delta(\varepsilon)$. Here we also applied the fact that the $C^{k,\gamma}$ norms defined using all metrics in $\mathcal V^{k,\gamma}_{2\varepsilon}$ are equivalent. This finishes the proof of the claim. 
\end{proof}

Let $[0,T_0)\subset[0,T-\varepsilon)$ be the maximum interval such that $g_t\in\mathcal{V}_{\varepsilon}^{k,\gamma}$ for all $t\in[0,T_0)$. If $T_0< T-\varepsilon$, then by Lemma \ref{longtimecontinuity} above, we have $g_{T_0}\in\mathcal{V}^{k,\gamma}_{\frac{\varepsilon}{2}}$. Using the argument as in the proof of Lemma \ref{shorttimestability}, we can find a positive number $\eta>0$, depending on $\varepsilon$ and $T_0$, such that $g_t\in\mathcal{V}_{\varepsilon}^{k,\gamma}$ for all $t\in[0,T_0+\eta]$, and this contradicts the definition of $T_0$. Hence we have $g_t\in\mathcal{V}_{\varepsilon}^{k,\gamma}$ for all $t\in[0,T-\varepsilon)$.

If $T=\infty$, then we may compute using the \L ojasiewicz inequality (Theorem \ref{Lojaciewicz}):
\begin{eqnarray*}
\frac{d}{dt}\left(\mu_{g_t}-\mu_{g_o}\right)^{1-2\alpha}&=&-2(1-2\alpha)\left(\mu_{g_t}-\mu_{g_o}\right)^{-2\alpha}\left\|\nabla\mu_{g_t}\right\|^2_{L^2_{g_t}}
\\
&\geq& 2(2\alpha-1)\left(\mu_{g_t}-\mu_{g_o}\right)^{-2\alpha}\cdot C\left(\mu_{g_t}-\mu_{g_o}\right)^{2\alpha}
\\
&\geq& 2C(2\alpha-1).
\end{eqnarray*}
Since $\mu_{g_t}\geq \mu_{g_o}$ for all $t\in[0,\infty)$, by integrating the above inequality, we have
\begin{equation}\label{mu decay}
\mu_{g_t}-\mu_{g_o}\leq C(\alpha)t^{-\frac{1}{2\alpha-1}}\quad\text{ for all }\quad t>0.
\end{equation}
Hence, by (\ref{infralapsarianism}), we have
\begin{equation*}
    \| g_{t_2} - g_{t_1}\|_{C^{k,\gamma}_{g_o}}\leq C\left( \mu_{g_{t_1}}- \mu_{g_o}\right)^{1-(2-\beta)\alpha}\leq C t_1^{-\frac{1-(2-\beta)\alpha}{2\alpha-1}}\quad\text{ for all }\quad 0<t_1<t_2<\infty.
\end{equation*}
This finishes the proof of the theorem.

\section{\texorpdfstring{$\nu$}{nu}-functional of closed manifolds with positive scalar curvature}

The following proposition is a straightforward and well-understood consequence of the Sobolev inequality. Since we have not found it in the literature, we shall include the proof for the convenience of the reader.

\begin{Proposition}\label{lowerbound of nu}
Let $(M^n,g)$ be a closed Riemannian manifold satisfying the following Sobolev inequality
\begin{eqnarray*}
\left(\int_M|u|^{\frac{2n}{n-2}}dg\right)^{\frac{n-2}{n}}\leq A\int_M|\nabla u|^2dg+B\int_M u^2dg\quad\text{ for all }\quad u\in W^{1,2}(M).
\end{eqnarray*}
Furthermore, assume $R\geq c_0>0$ everywhere on $M$. Then we have
$$\nu(g)\geq -C,$$
where $C$ is a constant depending only on $n$, $A$, $B$, and $c_0$.
\end{Proposition}
\begin{proof}
It is well-known that a Sobolev inequality is equivalent to a family of logarithmic Sobolev inequalities; see, for instance, \cite[Theorem 4.2.1]{Zhq11}. 
Hence, for all $\epsilon>0$ and for all $v\in W^{1,2}(M)$ with $\int_Mv^2dg=1$ we have 
\begin{eqnarray*}
\int_Mv^2\log v^2dg\leq \epsilon^2\int_M|\nabla v|^2dg-\frac{n}{2}\log \epsilon^2+BA^{-1}\epsilon^2+\frac{n}{2}\log\frac{nA}{2e}.
\end{eqnarray*}
By the lower bound of $R$, we have $\displaystyle \int_M Rv^2dg\geq c_0$. Furthermore, we can always take $A$ to be larger. Hence we may assume that $4B(Ac_0)^{-1}\leq1$. Then, for all $v\in W^{1,2}(M)$ with $\int_M v^2dg=1$, we have
\begin{eqnarray*}
\int_Mv^2\log v^2dg &\leq& \frac{\epsilon^2}{4}\left(\int_M4|\nabla v^2|dg+ 4BA^{-1}\right)-\frac{n}{2}\log \epsilon^2+\frac{n}{2}\log\frac{nA}{2e}.
\\
&\leq&\frac{\epsilon^2}{4}\left(\int_M4|\nabla v^2|dg+ 4B(Ac_0)^{-1}\int_M Rv^2dg\right)-\frac{n}{2}\log\left( 4\pi\cdot\frac{\epsilon^2}{4}\right)+\frac{n}{2}\log\frac{nA\pi}{2e}.
\\
&\leq&\tau\int_M\left(4|\nabla v|^2+Rv^2\right)dg-\frac{n}{2}\log(4\pi\tau)-n+C(n,A,B,c_0),
\end{eqnarray*}
where we have let $\tau:=\frac{\epsilon^2}{4}$. This obviously is equivalent to
$$\mu(g,\tau)\geq -C(n,A,B,c_0)$$
and the conclusion follows.
\end{proof}

\bibliography{bibliography}{}
\bibliographystyle{amsalpha}

\newcommand{\etalchar}[1]{$^{#1}$}
\providecommand{\bysame}{\leavevmode\hbox to3em{\hrulefill}\thinspace}
\providecommand{\MR}{\relax\ifhmode\unskip\space\fi MR }
\providecommand{\MRhref}[2]{%
  \href{http://www.ams.org/mathscinet-getitem?mr=#1}{#2}
}
\providecommand{\href}[2]{#2}

\bigskip
\bigskip

\noindent Department of Mathematics, University of California, San Diego, CA, 92093
\\ E-mail address: \verb"pachan@ucsd.edu "
\\

\noindent Department of Mathematics, University of California, San Diego, CA, 92093
\\ E-mail address: \verb"zim022@ucsd.edu"
\\

\noindent School of Mathematics, University of Minnesota, Twin Cities, MN, 55414
\\ E-mail address: \verb"zhan7298@umn.edu"

\end{document}